\documentclass{article}

\usepackage[colorlinks, citecolor=blue]{hyperref}

\usepackage{amscd}
\usepackage{times}
\usepackage{latexsym}
\usepackage{amsmath}
\usepackage{amssymb}
\usepackage{dsfont}
\usepackage{mathrsfs}
\usepackage{epsfig}
\usepackage{graphicx,graphics}
\usepackage{pstricks}
\usepackage{pst-all}
\usepackage[body={15cm,23cm}, top=2cm]{geometry}

\newtheorem{theorem}{Theorem}[section]
\newtheorem{corollary}[theorem]{Corollary}
\newtheorem{conjecture}[theorem]{Conjecture}
\newtheorem{definition}[theorem]{Definition}
\newtheorem{remark}[theorem]{Remark}
\newtheorem{proposition}[theorem]{Proposition}
\newtheorem{lemma}[theorem]{Lemma}
\newtheorem{example}[theorem]{Example}

\newenvironment{proof}[1][Proof]{\noindent\textbf{Proof:} }{\hfill\rule{3mm}{3mm}}

\newcommand{\mt}{\mathscr{T}}
\newcommand{\rank}{\mathop{\mathrm{rank}}}
\newcommand{\Ker}{\mathop{\mathrm{Ker}}}
\renewcommand{\Im}{\mathop{\mathrm{Im}}}
\newcommand{\lev}{\mathop{\mathrm{lev}}}
\begin{document}
\title {Dimensions of Biquadratic and Bicubic Spline Spaces \\over Hierarchical T-meshes}
\author{Chao Zeng \qquad Fang Deng \qquad Xin Li \qquad Jiansong Deng \\
\small $^1$ School of Mathematical Sciences,
University of Science and Technology of China, Hefei, China, 230026\\
\small  Email: zengchao@mail.ustc.edu.cn}
\date{}
\maketitle

\emph{\textbf{Abstract}
This paper discusses the dimensions of biquadratic $C^1$ spline spaces and bicubic $C^2$ spline spaces
over hierarchical T-meshes using the smoothing cofactor-conformality method. We obtain the dimension formula of
biquadratic $C^1$ spline spaces over hierarchical T-meshes in a concise way. In addition, we provide a dimension formula
for bicubic $C^2$ spline spaces over hierarchical T-mesh with fewer restrictions than that in  the previous literature.
A dimension formula for bicubic $C^2$ spline spaces over a new type hierarchical T-mesh is also provided.}


\section{Introduction}
\label{Introduction}
Splines are a useful tool for representing functions and surface models, and Non-uniform Rational B-Splines (NURBS),
which are defined on the tensor product meshes, are the most popular splines in the industry.
However, due to the tensor product structure, local refinement of
NURBS is impossible; furthermore, NURBS models generally contain a large number of superfluous control points.
Therefore, many splines defined on T-meshes are developed.

There are four main types of splines that are defined on T-meshes.
Hierarchical B-splines~\cite{For88} are polynomial splines defined on hierarchical T-meshes.
The definition is improved in \cite{Vuo11}. More papers \cite{Gia12,Mok140,Mok14} are published
to discuss the completeness and the partition of unity.
T-splines~\cite{Sed03,Sed04}
are rational splines defined on T-meshes, and they are a piecewise rational polynomial on every cell of the meshes.
The blending functions of T-splines may be linearly dependent, which is adverse to the application in IGA.
Therefore, analysis-suitable T-splines are introduced in \cite{Li12}.
Polynomial splines over hierarchical T-meshes (PHT-splines)~\cite{Deng08} are developed directly from the spline spaces.
The bases of PHT-splines are linearly independent and form a partition of unity. The main drawback of
PHT-splines is that they are only $C^1$ continuous.
LR-splines~\cite{Dok13}, which are defined on a special type of T-meshes, are also polynomial splines.
LR-splines are defined with the help of the dimension formulae, and they are complete.
However, the linear independence can not be guaranteed.
Among all of these splines, one fundamental problem is to
understand the spline space of the piecewise polynomials of a given smoothness on
a T-mesh, which is called the spline space over a T-mesh.

The spline space over a T-mesh $\mathbf{S}(m, n, \alpha, \beta, \mathscr{T})$ is first introduced in~\cite{Deng06},
which is a bi-degree $(m, n)$ piecewise
polynomial spline space over a T-mesh $\mt$ with smoothness order
$\alpha$ and $\beta$ in two directions. When $m\geqslant 2\alpha+1$ and $n\geqslant 2\beta+1$, a dimension formula
has been given in~\cite{Deng06}, and the bases have been constructed in~\cite{Deng08}. However,
if we relax the constraints of $m\geqslant 2\alpha+1$ and $n\geqslant 2\beta+1$, there is not a general dimension formula.
In 2011, \cite{Li11} discovered
that the dimension of the associated spline space has instability over some particular T-meshes,
i.e., the dimension is associated not only with the topological information of the T-mesh but also
with the geometric information of the T-mesh. In additon, in~\cite{Ber12}, D. Berdinskya
et al. give two more examples of $\mathbf{S}(5,5,3,3,\mathscr{T})$ and $\mathbf{S}(4,4,2,2,\mathscr{T})$ for the instability of
dimensions. These results suggest that we should consider the dimension formula for the spline space over some
special T-meshes. For this purpose, weighted T-meshes~\cite{Mou14}, diagonalizable T-meshes~\cite{Lipre}, and
T-meshes for hierarchical B-spline~\cite{Gia13}, over which the dimensions are stable, are developed .
The present paper is interested in hierarchical T-meshes, which have a nature tree structure and have existed in
the finite element analysis community for a long time. For a hierarchical T-mesh, \cite{Deng13} derives a
dimension formula for biquadratic $C^1$ spline spaces. And \cite{Wu13}
provides the dimension formula for $\mathbf{S}(d,d,d-1,d-1,\mathscr{T})$ over a very special hierarchical
T-mesh using the homological algebra technique.

In this paper, we present three results:
\begin{enumerate}
\item We provide a concise proof of the dimension formula of $\mathbf{S}(2,2,1,1,\mathscr{T})$ over a hierarchical T-mesh.
\item We prove a dimension formula for $\mathbf{S}(3,3,2,2,\mathscr{T})$ over a hierarchical T-mesh with fewer restrictions
than that in the previous literature.
\item
We give a dimension formula for $\mathbf{S}(3,3,2,2,\mathscr{T})$ over a new type of hierarchical
T-mesh that splits each cell into $3\times3$ parts during the local refinement.
\end{enumerate}

\begin{figure}[!ht]
\begin{center}
\begin{pspicture}(0,0)(4,4)
\psline(0,0)(4,0)(4,4)(0,4)(0,0)
\psline(0,1)(4,1)\psline(0,2)(4,2)\psline(0,3)(4,3)
\psline(1,0)(1,4)\psline(2,0)(2,4)\psline(3,0)(3,4)
\psline(1,1.5)(3,1.5)\psline(1,2.5)(3,2.5)\psline(1.5,1)(1.5,3)\psline(2.5,1)(2.5,3)
\psline(1,1.25)(2,1.25)\psline(1,1.75)(2.5,1.75)\psline(1.5,2.25)(2.5,2.25)
\psline(1.25,1)(1.25,2)\psline(1.75,1)(1.75,2.5)\psline(2.25,1.5)(2.25,2.5)
\end{pspicture}
\caption{A hierarchical T-mesh. \label{exa-dim}}
\end{center}
\end{figure}
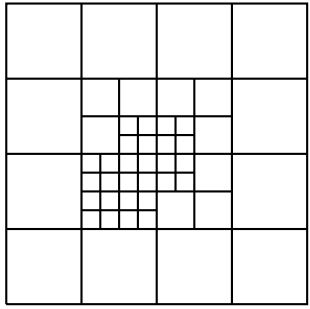

\cite{Lipre,Mok14,Wu13} also discuss the dimensions of $\mathbf{S}(3,3,2,2,\mathscr{T})$. For \cite{Lipre,Wu13},
their conclusions are aimed at diagonalizable T-meshes or weighted T-meshes. In this paper, the T-meshes we considered
may be not of these types. See Figure \ref{exa-dim} for an example.
For \cite{Mok14}, the authors present a new type of splines, which is called TDHB-splines.
The idea of this paper is to construct the bases of TDHB-splines first, then to prove the completeness of the bases.
That is to say, the dimension of the spline space has been obtained, which is the number of the bases.
For $\mathbf{S}(3,3,2,2,\mathscr{T})$, their conclusion is our third result; moreover,
they do not derive an explicit dimension formula.

The rest of the paper is organized as follows. In Section 2, we review the definitions and some results regarding T-meshes
and spline spaces over T-meshes. In Section 3, the smoothing cofactor-conformality method is explored. In section 4, we present
some conclusions about homogeneous boundary conditions. We discuss the dimensions
of $\mathbf{S}(2,2,1,1,\mathscr{T})$ over hierarchical T-meshes in Section 5 and the dimensions
of $\mathbf{S}(3,3,2,2,\mathscr{T})$ over hierarchical
T-meshes in Section 6. Section 7 ends the paper with conclusions and future work.

\section{T-Meshes and Spline Spaces}

A T-mesh is a rectangular grid that allows T-junctions.

\begin{definition}[\cite{Deng06,Deng13}]
Suppose $\mathscr{T}$ is a set of axis-aligned rectangles and the intersection of any two distinct rectangles in $\mathscr{T}$ either is
empty or consists of points on the boundaries of the rectangles. Then, $\mathscr{T}$ is called a \textbf{T-mesh}. Furthermore, if the entire
domain occupied by $\mathscr{T}$ is a rectangle, $\mathscr{T}$ is called a \textbf{regular T-mesh}.
If some edges of $\mathscr{T}$ also form a T-mesh $\mathscr{T}'$, $\mathscr{T}'$ is called a \textbf{submesh} of $\mathscr{T}$.
\end{definition}
In this paper, the T-meshes that we consider are regular, and we adopt the definitions of vertex, edge, and cell provided
in~\cite{Deng06}. T-meshes are allowed to have \textbf{T-junctions}, or \textbf{T-nodes},
which are the vertex of one rectangle that lies in the interior of an edge of another rectangle.
A valence 4 vertex is called a \textbf{crossing-vertex}.

Vertices, edges and cells can be divided into two parts. If a vertex is on the boundary grid
line of the T-mesh, it is called a \textbf{boundary vertex}.
Otherwise, it is called an \textbf{interior vertex}. There are two types of interior vertices: T-junctions and crossing-vertices.
If an edge is on the boundary of the T-mesh,
it is called a \textbf{boundary edge}; otherwise, it is
called an \textbf{interior edge}. A cell is called an \textbf{interior cell} if all its edges are interior edges; otherwise, it is called
a \textbf{boundary cell}.
An \textbf{l-edge} is a line segment that consists of several edges. It is the longest
possible line segment, the interior edges of which are connected and the two
end points are T-junctions or boundary vertices. If an l-edge is comprised of some boundary edges, the l-edge is called a \textbf{boundary l-edge};
otherwise, it is called an \textbf{interior l-edge}. There are three types of interior l-edges.
If the two end points of an interior l-edge are both T-junctions, the l-edge is called a \textbf{T l-edge}.
If the two end points of an interior l-edge are both boundary vertices, the l-edge is called a \textbf{cross-cut}.
If one end point is a boundary vertex and the other one is an interior vertex, the l-edge is called
a \textbf{ray}.

\begin{figure}[!ht]
\begin{center}
\psset{unit=0.75cm,linewidth=0.8pt}
\begin{pspicture}(-0.5,-0.5)(5.5,6.5)
{\psset{linewidth=0.01pt}
\psframe[fillstyle=solid,fillcolor=lightgray](0,0)(1,1)
\psframe[fillstyle=solid,fillcolor=lightgray](3,3)(4,5)}
\psline(0,0)(5,0)(5,6)(0,6)(0,0)
\psline(0,1)(5,1)\psline(1,2)(4,2)\psline(3,3)(4,3)\psline(1,4)(3,4)\psline(0,5)(5,5)
\psline(1,0)(1,6)\psline(2,1)(2,5)\psline(3,1)(3,6)\psline(4,0)(4,6)
\psdots(0,0)\rput(-0.4,-0.3){$v_7$}\psdots(5,0)\rput(5.4,-0.3){$v_8$}
\psdots(1,2)\rput(0.6,2){$v_5$}\psdots(4,2)\rput(4.4,2){$v_6$}
\psdots(3,1)\rput(3,0.6){$v_4$}\psdots(3,6)\rput(3,6.4){$v_3$}
\psdots(0,5)\rput(-0.4,5){$v_1$}\psdots(5,5)\rput(5.4,5){$v_2$}
\psdots(2,4)\rput(1.7,3.7){$v_9$}\psdots(3,2)\rput(2.7,2.3){$v_{10}$}
\rput(0.5,0.5){$c_1$}\rput(3.5,4){$c_2$}
\end{pspicture}
\caption{A regular T-mesh. \label{T-mesh}}
\end{center}
\end{figure}
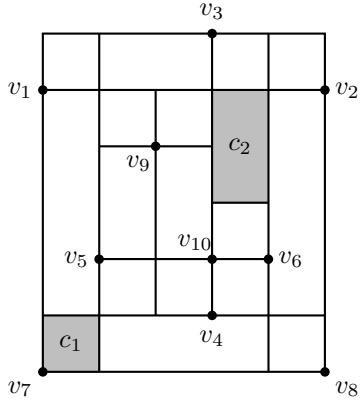
In Figure \ref{T-mesh}, $v_1,v_2,v_3,v_7$ and $v_8$ are boundary vertices, $v_9$ and $v_{10}$ are crossing-vertices, and $v_4,v_5$ and $v_6$
are T-junctions; $c_1$ is a boundary cell, and $c_2$ is an interior cell;
the l-edge between $v_1$ and $v_2$ is a cross-cut, the l-edge between $v_3$ and $v_4$ is a ray, the l-edge between $v_5$ and $v_6$ is a T l-edge, and the l-edge between $v_7$ and $v_8$ is a boundary l-edge.

\subsection{Hierarchical T-meshes}\label{hierarchy}

A hierarchical T-mesh \cite{Deng08} is a special type of T-mesh
that has a natural level structure. It is defined recursively.
Generally, we start from a tensor-product mesh (level
$0$), the elements (vertices, edges and cells) of which are called level $0$ elements. Then, some cells at level $k$ are
each divided into $m\times n$ subcells equally, where the new vertices, the new edges and the new cells are of level $k+1$. The resulting T-mesh
is  called \textbf{a hierarchical T-mesh of $m \times n$ division}.
In the following, the default division of a hierarchical T-mesh is $2\times 2$.
Figure \ref{HTmesh} illustrates the process of generating a hierarchical T-mesh.
To emphasize the level structure of a hierarchical T-mesh $\mathscr{T}$, in certain cases,
we denote the T-mesh of level $k$ as $\mathscr{T}^k$.
The maximal level number that appears is
called the \textbf{level} of the hierarchical T-mesh, and is denoted by $\lev({\mathscr{T}})$.

\begin{figure}[!ht]
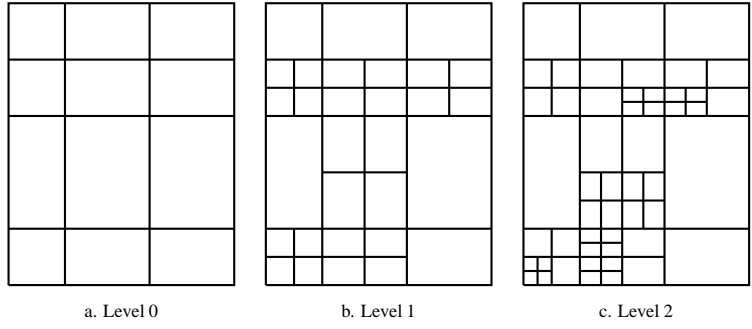

\begin{center}
\psset{unit=0.75cm,linewidth=0.8pt}
\begin{tabular}{ccc}
\pspicture(0,0)(4,5) \psline(0,0)(4,0)(4,5)(0,5)(0,0)
\psline(0,1)(4,1)\psline(0,3)(4,3)\psline(0,4)(4,4)
\psline(1,0)(1,5)\psline(2.5,0)(2.5,5)
\endpspicture &
\pspicture(0,0)(4,5)
\psline(0,0)(4,0)(4,5)(0,5)(0,0)\psline(0,1)(4,1)\psline(0,3)(4,3)\psline(0,4)(4,4)
\psline(1,0)(1,5)\psline(2.5,0)(2.5,5)\psline(1.75,0)(1.75,4)
\psline(0,0.5)(2.5,0.5)\psline(0.5,0)(0.5,1)\psline(2.5,2)(1,2)\psline(0,3.5)(4,3.5)
\psline(0.5,3)(0.5,4)\psline(3.25,3)(3.25,4)
\endpspicture &\pspicture(0,0)(4,5) \psline(0,0)(4,0)(4,5)(0,5)(0,0)
\psline(0,1)(4,1)\psline(0,3)(4,3)\psline(0,4)(4,4)
\psline(1,0)(1,5)\psline(2.5,0)(2.5,5)\psline(1.75,0)(1.75,4)
\psline(0,0.5)(2.5,0.5)\psline(0.5,0)(0.5,1)\psline(2.5,2)(1,2)\psline(0,3.5)(4,3.5)
\psline(0.5,3)(0.5,4)\psline(3.25,3)(3.25,4)\psline(0,0.25)(0.5,0.25)
\psline(0.25,0)(0.25,0.5)\psline(1.375,0)(1.375,2)\psline(1,0.25)(1.75,0.25)
\psline(1,0.75)(1.75,0.75)\psline(1,1.5)(2.5,1.5)\psline(2.125,1)(2.125,2)
\psline(2.125,3)(2.125,3.5)\psline(1.75,3.25)(3.25,3.25)\psline(2.875,3)(2.875,3.5)
\endpspicture \\
\scriptsize  a. Level 0  & \scriptsize   b. Level 1  & \scriptsize  c. Level 2
\end{tabular}
\caption{A hierarchical T-mesh. \label{HTmesh}}
\end{center}
\end{figure}

For a hierarchical T-mesh $\mathscr{T}$ of level $n$, we use $T_k$ to denote the set of all of the level $k$ l-edges and
the level $k$ vertices of $\mathscr{T}$,
and $T_k^o$ to denote the set of all of the level $k$ l-edges and all of the vertices on these level $k$ l-edges in $\mathscr{T}$. For any l-edge $l$ of level
$j$, we use $N(l)$ to denote the number of the cells of $\mathscr{T}^{j-1}$ that $l$ crosses.

Figure \ref{exam1} is an example. $\mathscr{T}$ is a hierarchical T-mesh and $\lev(\mathscr{T})=2$. The level 1 l-edges are labeled with dashed lines,
the level 1 vertices are labeled with``$\bullet$", and the level 2 vertices are labeled with ``$\circ$".
Here, $N(l_1)=2$ and $N(l_2)=1$ for the two l-edges $l_1$(between $v_1$ and $v_2$) and $l_2$(between $v_3$ and $v_4$).

\begin{figure}[!htp]
\begin{center}
\begin{tabular}{c@{\hspace*{1.3cm}}c@{\hspace*{1cm}}c}
\begin{pspicture}(0,0)(3,3)
\psline(0,0)(3,0)(3,3)(0,3)(0,0)\psline(1,0)(1,3)\psline(2,0)(2,3)\psline(0,1)(3,1)\psline(0,2)(3,2)
\psline(1.5,1.25)(2,1.25)\psline(1.75,1)(1.75,1.5)
\psset{linestyle=dashed}
\psline(1,1.5)(3,1.5)\psline(1.5,0)(1.5,2)\psline(1,0.5)(3,0.5)\psline(2.5,0)(2.5,2)
\psdots[dotscale=0.7](1,0.5)\rput(0.8,0.5){$v_1$}\psdots[dotscale=0.7](3,0.5)\rput(3.25,0.5){$v_2$}
\rput(1.75,0.8){$v_4$}\rput(1.75,1.7){$v_3$}
\psdots[dotscale=0.7](2,1.5)(1.5,1)(2,0.5)(2.5,1)(1,1.5)(3,1.5)(1.5,0)(1.5,2)(2.5,0)(2.5,2)(1.5,0.5)(2.5,0.5)(1.5,1.5)(2.5,1.5)
\psdots[dotscale=0.7,dotstyle=o](1.5,1.25)(1.75,1.5)(1.75,1)(2,1.25)
\end{pspicture} &
\begin{pspicture}(0,0)(2,2)
\psset{linestyle=dashed}
\psline(0,1.5)(2,1.5)\psline(0.5,0)(0.5,2)\psline(0,0.5)(2,0.5)\psline(1.5,0)(1.5,2)
\psdots[dotscale=0.7](1,1.5)(0.5,1)(1,0.5)(1.5,1)(0,1.5)(2,1.5)(0.5,0)(0.5,2)(0,0.5)(2,0.5)(1.5,0)(1.5,2)(0.5,0.5)(1.5,0.5)(0.5,1.5)(1.5,1.5)
\end{pspicture} &
\begin{pspicture}(0,0)(2,2)
\psset{linestyle=dashed}
\psline(0,1.5)(2,1.5)\psline(0.5,0)(0.5,2)\psline(0,0.5)(2,0.5)\psline(1.5,0)(1.5,2)
\psdots[dotscale=0.7,dotstyle=o](0.5,1.25)(0.75,1.5)
\psdots[dotscale=0.7](1,1.5)(0.5,1)(1,0.5)(1.5,1)(0,1.5)(2,1.5)(0.5,0)(0.5,2)(0,0.5)(2,0.5)(1.5,0)(1.5,2)(0.5,0.5)(1.5,0.5)(0.5,1.5)(1.5,1.5)
\end{pspicture} \\
$\mathscr{T}$ & $T_1$ & $T_1^o$
\end{tabular}
\caption{\label{exam1}$T_k$ and $T_k^o$.}
\end{center}
\end{figure}
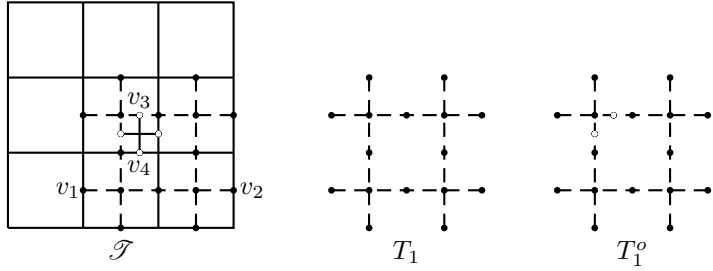

\subsection{Spline spaces over T-meshes}
Given a T-mesh $\mathscr{T}$, we use $\mathscr{F}$ to denote all of the cells in $\mathscr{T}$ and $\Omega$ to denote the region
occupied by the cells in $\mathscr{T}$. In \cite{Deng06}, the following spline space definition is proposed:
$$
\mathbf{S}(m,n,\alpha,\beta,\mathscr{T}):=\{f(x,y)\in
C^{\alpha,\beta}(\Omega): f(x,y)|_{\phi}\in \mathbb{P}_{mn},\forall
\phi\in\mathscr{F}\},
$$
where $\mathbb{P}_{mn}$ is the space of the polynomials with bi-degree $(m,n)$ and $C^{\alpha,\beta}$ is the space consisting of
all of the bivariate functions continuous in $\Omega$ with order $\alpha$ along the $x$-direction and with order $\beta$
along the $y$-direction. It is obvious that $\mathbf{S}(m,n,\alpha,\beta,\mathscr{T})$ is a linear space. In this paper, we only discuss $\mathbf{S}(d,d,d-1,d-1,\mathscr{T})$, which is denoted as $\mathbf{S}^{d}(\mathscr{T})$ for convenience.

\section{Smoothing Cofactor-conformality Method and Conformality Vector Spaces}\label{smocof}
In this section, we review the smooth cofactor-conformality method introduced in~\cite{Sch}
and~\cite{Wang} for computing the dimensions of spline spaces over T-meshes~\cite{LiC,Lipre,Wu13}.

\begin{figure}[!htp]
\begin{center}
\begin{pspicture}(0,0)(3,3)
\psline(0,1.5)(3,1.5)\psline(1.5,0)(1.5,3)
\psdots[dotscale=0.7](1.5,1.5)\rput(2.3,1.2){$v_i(x_i,y_i)$}
\rput(2.25,2.25){$f_4$}\rput(2.25,0.75){$f_3$}
\rput(0.75,2.25){$f_1$}\rput(0.75,0.75){$f_2$}
\end{pspicture}
\caption{Smoothing conditions in adjacent cells.\label{smooth}}
\end{center}
\end{figure}
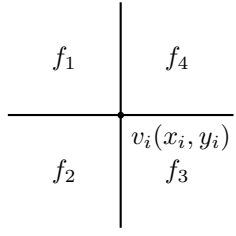

Referring to Figure~\ref{smooth}, let $f_j(x,y), j = 1, 2, 3, 4$ be the
bivariate polynomials surrounding the interior vertex $v_{i}(x_i,y_i)$ (if the vertex $v_{i}$ is a T-junction,
some of the polynomials are identical). Then there exist a constant $\gamma_{i}\in \mathbb{R}$ and two polynomials $a(y)\in \mathbb{P}_d[y], b(x)\in \mathbb{P}_d[x]$, such that
\begin{align*}
& f_1(x,y)-f_2(x,y)=b(x)(y-y_i)^d, \\
& f_3(x,y)-f_2(x,y)=a(y)(x-x_i)^d,  \\
& f_4(x,y)-f_2(x,y)=a(y)(x-x_i)^d+b(x)(y-y_i)^d+\gamma_{i}(x-x_i)^d(y-y_i)^d,
\end{align*}
where $a(y)$ and $b(x)$ are the edge co-factors associated with the corresponding edges, and $\gamma_{i}$ is the vertex co-factor associated with vertex $v_{i}$.

There are other constraints for the T l-edges of the T-mesh. For example, given a horizontal T l-edge $l_j$ with $r$
vertices $v_{j_{1}},v_{j_{2}},\ldots,v_{j_{r}}$, let the $x$-coordinate of $v_{j_{i}}$ be $x_{j_{i}}$, and the vertex co-factors of $v_{j_{i}}$ be $\gamma_{j_{i}}$.
Then we have
$$
\sum_{i=1}^r \gamma_{j_{i}}(x-x_{j_{i}})^d=0.
$$
This equation is equivalent to a linear system denoted by $\mathscr{S}_{l_j}=0$:
\begin{equation}\label{system}
\left\{
\begin{array}{lll}
\sum_{i=1}^r \gamma_{j_{i}}=0,  \\
\sum_{i=1}^r \gamma_{j_{i}}x_{j_{i}}=0,  \\
\cdots,  \\
\sum_{i=1}^r \gamma_{j_{i}}x_{j_{i}}^d=0.
\end{array}
\right.
\end{equation}
Similarly, we can derive the linear system for a vertical T l-edge.

As in~\cite{Wu13}, we can define the {\bf conformality vector space} for a set of T l-edges as follows.
\begin{definition}
Suppose $L$ is a set of T l-edges: $L=\{l_1,l_2,\ldots,l_n:l_i \mbox{ is a T l-edge,} 1\leqslant i \leqslant n\}$,
$v_1,v_2,\ldots,v_m$ are all of the vertices on $l_1,l_2,\ldots,l_n$, and $\gamma_j$ is the co-factor of $v_j$. Then the
\textbf{conformality vector space} $W[L]$ of $L$ is defined by
$$
W[L]:=\{\gamma=(\gamma_1, \gamma_2,\ldots,\gamma_m)^T:\mathscr{S}_{l_i}=0,\,0\leqslant i\leqslant n\},
$$
where $\mathscr{S}_{l_i}=0$ is the linear system as Equation \eqref{system} associated with
the l-edge $l_i$.
For some predefined order of the vertex co-factors and the l-edges, the coefficient
matrix for the homogeneous system of $W[L]$ is called the \textbf{conformality matrix} of $L$.
\end{definition}

In particular, if $L$ only contains one T l-edge $l_1$ with $V_0$ vertices,
$\dim W[l_1]=(V_0-d-1)_+:=\mbox{max}(0,V_0-d-1)$. In the following, we also call such an l-edge an l-edge of order $(V_0-d-1)_+$.
When $(V_0-d-1)_+=0$, we say $l_1$ is a \textbf{trivial} l-edge.

For an l-edge of order 1, we have the following property.
\begin{proposition}
In a given T-mesh $\mathscr{T}$, suppose $l$ is a T l-edge with $d+2$ vertices $v_0,v_1,\ldots,v_{d+1}$ (from one end point to the other end point) and the co-factor of $v_i$ is $\gamma_i$. Then
\begin{align} \label{1-order}
\gamma_{i+k}=(-1)^k \gamma_i \frac{\prod_{j\neq i} d_{i,j}}{\prod_{j\neq i+k} d_{i+k,j}},
\end{align}
where $d_{i,j}$ is the distance between $v_i$ and $v_j$, $k\geqslant 1$.
\end{proposition}
\begin{proof}
Suppose $l$ is a horizontal T l-edge, and the $x$-coordinate of $v_i$ is $x_i$. Then the linear
system $\mathscr{S}_{l}=0$ has the following form:
$$
\begin{pmatrix}
1    & 1    & \cdots & 1    & \cdots & 1    & \cdots  & 1   \\
x_0  & x_1  & \cdots & x_i  & \cdots & x_{i+k} & \cdots & x_{d+1} \\
\vdots & \vdots &   & \vdots &      &  \vdots &     & \vdots \\
x_0^d  & x_1^d  & \cdots & x_i^d  & \cdots & x_{i+k}^d & \cdots & x_{d+1}^d
\end{pmatrix}
\begin{pmatrix}
\gamma_0 \\ \gamma_1 \\ \vdots \\ \gamma_{d+1}
\end{pmatrix}
=0
$$

$$
\Rightarrow
\begin{pmatrix}
1    & \cdots & 1    & 1  & \cdots  & 1   \\
x_0  & \cdots & x_{i-1}   & x_{i+1} & \cdots & x_{d+1} \\
\vdots  &  & \vdots & \vdots &     & \vdots \\
x_0^d   & \cdots & x_{i-1}^d   & x_{i+1}^d & \cdots & x_{d+1}^d
\end{pmatrix}
\begin{pmatrix}
\gamma_0 \\ \vdots \\ \gamma_{i-1} \\ \gamma_{i+1} \\\vdots \\ \gamma_{d+1}
\end{pmatrix}
=-\gamma_i
\begin{pmatrix}
1 \\ x_i \\ \vdots \\ x_i^d
\end{pmatrix}.
$$
According to Cramer's Rule of the linear equation system, we know that $\gamma_{i+k}=-\gamma_i \cdot \frac{\Delta'}{\Delta}$. Here
$$
\Delta =
\begin{vmatrix}
1    & \cdots & 1    & 1  & \cdots  & 1   \\
x_0  & \cdots & x_{i-1}   & x_{i+1} & \cdots & x_{d+1} \\
\vdots  &  & \vdots & \vdots &     & \vdots \\
x_0^d   & \cdots & x_{i-1}^d   & x_{i+1}^d & \cdots & x_{d+1}^d
\end{vmatrix},
$$
$$
\Delta' =
\begin{vmatrix}
1    & \cdots & 1    & 1  & \cdots & 1 & 1 & 1 & \cdots  & 1   \\
x_0  & \cdots & x_{i-1}   & x_{i+1} & \cdots & x_{i+k-1} & x_i & x_{i+k+1} & \cdots & x_{d+1} \\
\vdots  &  & \vdots & \vdots & & \vdots & \vdots & \vdots &     & \vdots \\
x_0^d   & \cdots & x_{i-1}^d   & x_{i+1}^d & \cdots & x_{i+k-1}^d & x_i^d &x_{i+k+1}^d & \cdots & x_{d+1}^d
\end{vmatrix}.
$$
$\Delta$ and $\Delta'$ are both Vandermonde determinants. Hence, it is easy to verify that equation \eqref{1-order} holds.
\end{proof}

Referring to \cite{LiC}, we have the following lemma for the dimensions of the spline spaces over T-meshes.
\begin{lemma}
\label{lem2.1}
Given a T-mesh $\mathscr{T}$, suppose it has $n_{c}$ cross-cuts and $n_{v}$ interior vertices and $M$ is the conformality matrix of all of the
T l-edges.  Then, it follows that
$$
\dim \mathbf{S}^d(\mathscr{T})=(d+1)^2+n_{c}(d+1)+ n_{v}-\rank M.
$$
\end{lemma}

\begin{remark}For a given T-mesh $\mathscr{T}$, we use $L(\mathscr{T})$ to denote the set of all of the
T l-edges in $\mathscr{T}$. Then, Lemma \ref{lem2.1} states that
\begin{align} \label{equ2.1}
\dim \mathbf{S}^d(\mathscr{T})=\dim W[L(\mathscr{T})]+\dim \mathbf{S}^d(\mathscr{T}\backslash L(\mathscr{T})).
\end{align}
Here, $\mathscr{T}\backslash L(\mathscr{T})$ is the mesh obtained by deleting $L(\mathscr{T})$ from $\mathscr{T}$. See
Figure~\ref{dim} for an example.
\end{remark}

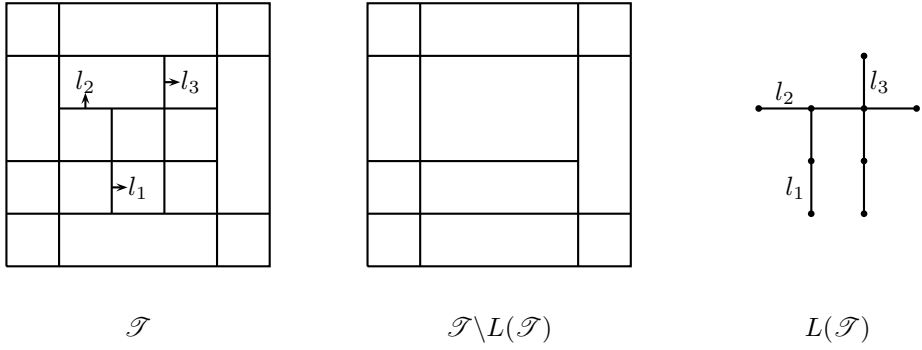
\begin{figure}[!ht]
\begin{center}
\psset{unit=0.7cm,linewidth=0.8pt}
\begin{tabular}{c@{\hspace*{1.3cm}}c@{\hspace*{1cm}}c}
\begin{pspicture}(0,0)(5,5)
\psline(0,0)(0,5)(5,5)(5,0)(0,0)
\psline(0,1)(5,1)\psline(0,2)(4,2)\psline(1,3)(4,3)\psline(0,4)(5,4)
\psline(1,0)(1,5)\psline(2,1)(2,3)\psline(3,1)(3,4)\psline(4,0)(4,5)
\psline{->}(2,1.5)(2.3,1.5)\psline{->}(1.5,3)(1.5,3.3)\psline{->}(3,3.5)(3.3,3.5)
\rput(2.5,1.5){$l_1$}\rput(1.5,3.5){$l_2$}\rput(3.5,3.5){$l_3$}
\end{pspicture} &
\begin{pspicture}(0,0)(5,5)
\psline(0,0)(0,5)(5,5)(5,0)(0,0)
\psline(0,1)(5,1)\psline(0,2)(4,2)\psline(0,4)(5,4)
\psline(1,0)(1,5)\psline(4,0)(4,5)
\end{pspicture} &
\begin{pspicture}(0,0)(5,5)
\psline(1,3)(4,3)\psline(2,1)(2,3)\psline(3,1)(3,4)
\psdots[dotscale=0.7](1,3)(2,3)(3,3)(4,3)(2,1)(2,2)(3,1)(3,2)(3,4)
\rput(1.7,1.5){$l_1$}\rput(1.5,3.3){$l_2$}\rput(3.3,3.5){$l_3$}
\end{pspicture}  \\[0.5cm]
$\mathscr{T}$ & $\mathscr{T}\backslash L(\mathscr{T})$ & $L(\mathscr{T})$
\end{tabular}
\caption{Decompose a T-mesh into a semi-cross mesh and the set of all of the T l-edges.\label{dim}}
\end{center}
\end{figure}

\begin{lemma}
\label{2.3}
For a given T-mesh $\mathscr{T}$, suppose $L(\mathscr{T})$ is the set of all of the T l-edges in $\mathscr{T}$, and
$L$ is a subset of $L(\mathscr{T})$. Then, the following statements are equivalent:
\begin{enumerate}
  \item $\dim W[L(\mathscr{T})]=\dim W[L]+\dim W[L(\mathscr{T}\backslash L)]$.
  \item The projection mapping $\pi : W[L(\mathscr{T})]\rightarrow W[L]$ is surjective.
  \item Suppose the conformality matrix of $L$ is $A$, the conformality matrix of $L(\mathscr{T}\backslash L)$ is $B$, and the conformality matrix of $T(\mathscr{T})$ is $M$. Then, $\rank M = \rank A  + \rank B$.
\end{enumerate}
\end{lemma}
\begin{proof}Suppose $L(\mathscr{T})$ has $V$ interior vertices and
$L$ has $V'$ vertices. The matrix $M$ has the form
$\begin{pmatrix}A & \\C & B\end{pmatrix}$ if we define the order
of the vertex co-factors and the l-edges of $L(\mathscr{T})$ with that of $L$ in front of the others.
 \begin{enumerate}
  \item (1) $\Leftrightarrow$ (2): \\
  For the projection mapping $\pi : W[L(\mathscr{T})]\rightarrow W[L]$, it is apparent that
  $\Ker \pi=W[L(\mathscr{T}\backslash L)]$. Because $\dim W[L(\mathscr{T})]=\dim \Ker \pi +\dim \Im \pi$, (1) is equivalent to (2).

\item (3) $\Rightarrow$ (1): \\
  Because $\rank M = \rank A+\rank B$, we have
  \begin{align*}
  \dim W[L(\mathscr{T})]&= V - \rank M  \\
                          &= (V-V') - \rank B + (V'-\rank A)  \\
                          &= \dim W[L(\mathscr{T}\backslash L)]+\dim W[L].
  \end{align*}

  \item (1) $\Rightarrow$ (3):\\
  Suppose $\rank M \neq \rank A+ \rank B$. Then, $\rank M > \rank A +\rank B$. Thus,
  \begin{align*}
  \dim W[L(\mathscr{T})] & =  V - \rank M  \\
         &< (V-V')-\rank B +(V'-\rank A )  \\
         &= \dim W[L(\mathscr{T}\backslash L)]+\dim W[L],
  \end{align*}
  a contradiction.
 \end{enumerate}
\end{proof}

From the proof of lemma \ref{2.3}, we know that
\begin{equation}
\dim W[L(\mathscr{T})]\leqslant \dim W[L(\mathscr{T}\backslash L)]+\dim W[L].  \label{equ2.2}
\end{equation}
Conversely, we have
\begin{equation}
\dim W[L(\mathscr{T})]\geqslant \dim W[L(\mathscr{T}\backslash L)]. \label{lowbound}
\end{equation}
This is because $\rank M\leqslant \rank\begin{pmatrix}A \\ C\end{pmatrix}+\rank(B)\leqslant V'+\rank B$,
\begin{align*}
\dim W[L(\mathscr{T})]&=V-\rank M  \\
         &\geqslant (V-V')-\rank B \\
         &= \dim [L(\mathscr{T}\backslash L)].
\end{align*}

\begin{lemma}
\label{surject2}
For a given T-mesh $\mathscr{T}_1$, $L_{1}=\{l_1,l_2,\ldots,l_n\}$ is a subset of $L(\mathscr{T}_1)$. We extend $l_i$ to $l'_i$ and
get a new mesh $\mathscr{T}_2$, such that $l'_i$ is still a T l-edge of $\mathscr{T}_2$. Thus, $L_2=\{l'_1,l'_2,\ldots,l'_n\}$
is a subset of $L(\mathscr{T}_2)$.
As in Figure \ref{f-sur}, $L_1=\{l_1,l_3\}$, $L_2=\{l'_1,l'_3\}.$
If $\dim W[L(\mathscr{T}_2)]=\dim W[L_{2}]+\dim W[L(\mathscr{T}_2\backslash L_{2})]$,
then $\dim W[L(\mathscr{T}_1)]=\dim W[L_{1}]+ \dim W[L(\mathscr{T}_1\backslash L_{1})]$.
\end{lemma}

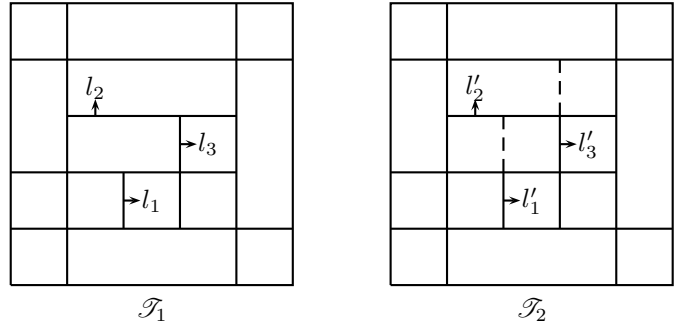
\begin{figure}[!ht]
\begin{center}
\psset{unit=0.75cm,linewidth=0.8pt}
\begin{tabular}{c@{\hspace*{1.3cm}}c}
\begin{pspicture}(0,0)(5,5)
\psline(0,0)(0,5)(5,5)(5,0)(0,0)
\psline(0,1)(5,1)\psline(0,2)(4,2)\psline(1,3)(4,3)\psline(0,4)(5,4)
\psline(1,0)(1,5)\psline(2,1)(2,2)\psline(3,1)(3,3)\psline(4,0)(4,5)
\psline{->}(2,1.5)(2.3,1.5)\psline{->}(1.5,3)(1.5,3.3)\psline{->}(3,2.5)(3.3,2.5)
\rput(2.5,1.5){$l_1$}\rput(1.5,3.5){$l_2$}\rput(3.5,2.5){$l_3$}
\end{pspicture} &
\begin{pspicture}(0,0)(5,5)
\psline(0,0)(0,5)(5,5)(5,0)(0,0)
\psline(0,1)(5,1)\psline(0,2)(4,2)\psline(1,3)(4,3)\psline(0,4)(5,4)
\psline(1,0)(1,5)\psline(2,1)(2,2)\psline(3,1)(3,3)\psline(4,0)(4,5)
\psline{->}(2,1.5)(2.3,1.5)\psline{->}(1.5,3)(1.5,3.3)\psline{->}(3,2.5)(3.3,2.5)
\psline[linestyle=dashed](2,2)(2,3)\psline[linestyle=dashed](3,3)(3,4)
\rput(2.5,1.5){$l_1'$}\rput(1.5,3.5){$l_2'$}\rput(3.5,2.5){$l_3'$}
\end{pspicture}
 \\
$\mathscr{T}_1$ & $\mathscr{T}_2$
\end{tabular}
\caption{Figure for lemma \ref{surject2}.\label{f-sur}}
\end{center}
\end{figure}

\begin{proof}
Suppose the conformality matrices of $L_{1}$, $L_{2}$, $L(\mathscr{T}_1)$ and $L(\mathscr{T}_2)$ are $A_{1}$, $A_{2}$, $M_{1}$ and
$M_{2}$, respectively. We can
arrange the order of the l-edges and the vertex co-factors such that $M_{1}=\begin{pmatrix}A_{1} & \\ C_{1} & B \end{pmatrix}$, and
$M_{2}=\begin{pmatrix}A_{2} & \\ C_{2} & B \end{pmatrix}$.
$A_{2}$ is obtained by adding some columns to $A_{1}$ on the vertices added from
$L_{1}$ to $L_{2}$, and so is $C_{2}$.

Because $\dim W[L(\mathscr{T}_2)]=\dim W[L_{2}]+\dim W[L(\mathscr{T}_2\backslash L_{2})]$, it follows that the projection mapping
$\pi' : W[L(\mathscr{T}_2)]\rightarrow W[L_{2}]$ is surjective, i.e., for any $x$ that satisfies $A_{2}x=0$, the solution
of $By=-C_{2}x$ exists. Then, for any $X_{1}$ that satisfies $A_{1}X_{1}=0$, we add $0$ entries to $X_{1}$ on the vertices added from $L_{1}$
to $L_{2}$. Suppose we obtain $X_{2}$ after adding. It is easy to verify that $A_{1}X_{1}=A_{2}X_{2}$ and $C_{1}X_{1}=C_{2}X_{2}$. Thus,
$ A_{1}X_{1}=0\Rightarrow A_{2}X_{2}=0\Rightarrow \mbox{the solution of } BY=-C_{2}X_{2}=-C_{1}X_{1} \mbox{ exists}$.
That is, the projection mapping $\pi : W[L(\mathscr{T}_1)]\rightarrow W[L_{1}]$ is surjective, which proves the lemma.
\end{proof}

\begin{lemma}
\label{surject1}
For a given T-mesh $\mathscr{T}$, $L$ is a subset of $L(\mathscr{T})$. Suppose $L(\mathscr{T}\backslash L)$ has $V$ vertices and $E$ l-edges.
If $\dim W[L(\mathscr{T}\backslash L)]=V-(d+1)E$, then $\dim W[L(\mathscr{T})]=\dim W[L]+\dim W[L(\mathscr{T}\backslash L)]$.
\end{lemma}

\begin{proof}
Suppose the conformality matrix of $L$ is $A$, the conformality matrix of $L(\mathscr{T}\backslash L)$ is B, and
the conformality matrix of $L(\mathscr{T})$ is $M$. Then, we can arrange the order of the vertex co-factors and the l-edges
such that $M=\begin{pmatrix}A & \\ C & B \end{pmatrix}$.

Because $\rank B =V-\dim W[L(\mathscr{T}\backslash L)]=V-(V-(d+1)E)=(d+1)E$, it follows that $B$ is of full row rank.
We have $\rank M \leqslant \rank A + \rank \begin{pmatrix}C & B \end{pmatrix} \leqslant \rank A + (d+1)E= \rank A + \rank B$.
Conversely, $\rank M \geqslant \rank A+ \rank B$. Thus, $\rank M= \rank A+ \rank B$, which is
equivalent to $\dim W[L(\mathscr{T})]=\dim W[L]+\dim W[L(\mathscr{T}\backslash L)]$ according to Lemma \ref{2.3}.
\end{proof}

Using lemma \ref{surject1}, we can obtain the following corollary.

\begin{corollary}
\label{general}
Given a T-mesh $\mathscr{T}$, if there is an order of all of the T l-edges, say $l_1,l_2,\ldots,l_n$, such that $n_{l_i}\geqslant d+1$, where $n_{l_i}$ is the number of the vertices on $l_i$ but not on $l_j, j=1,2,\ldots,i-1$, then $\dim W[L(\mathscr{T})]=V-(d+1)n$, where $V$ is the number of the vertices on all of these l-edges.
\end{corollary}

\begin{proof}
Let $L=\{l_1,l_2,\ldots,l_{n-1}\}$. Then, $L(\mathscr{T}\backslash L)$ has $n_{l_n}$ vertices. Because $n_{l_n}\geqslant d+1$, $\dim W[L(\mathscr{T}\backslash L)]=(n_{l_n}-(d+1))_+=n_{l_n}-(d+1)$. According to Lemma \ref{surject1}, we have $\dim W[L(\mathscr{T})]=\dim W[L]+\dim W[L(\mathscr{T}\backslash L)]=\dim W[L]+n_{l_n}-(d+1)$. $\dim W[L]$ can be analyzed as
$\dim W[L(\mathscr{T})]$. Continuing this process, at last we have
$$
\dim W[L(\mathscr{T})]=\sum_{i=1}^n (n_{l_i}-(d+1))=V-(d+1)n.
$$
\end{proof}

For any subset $L$ of $L(\mathscr{T})$, if it has at least one such order as Corollary~\ref{general}, we say $L$ has a \textbf{reasonable order}.

\section{Extended T-Meshes and Homogeneous Boundary Conditions}\label{exthbc}

For a T-mesh $\mathscr{T}$ of $\mathbf{S}(m,n,m-1,n-1,\mathscr{T})$, the extended T-mesh $\mathscr{T}^{\varepsilon}$ is an enlarged
T-mesh by copying each horizontal boundary l-edge of $\mathscr{T}$ $m$ times, and each vertical boundary l-edge of $\mathscr{T}$ $n$ times,
and by extending all of the l-edges with an end point on the boundary of $\mathscr{T}$. Figure \ref{extend} is an example of
an extended T-mesh for $\mathbf{S}^2(\mathscr{T})$.

\begin{figure}[!htb]
\begin{center}
\psset{unit=0.75cm,linewidth=0.8pt}
\begin{tabular}{cc}
\begin{pspicture}(0.5,0.5)(6,6)
\psline(1,1)(1,5)(5,5)(5,1)(1,1) \psline(1,2)(5,2)\psline(2,3)(4,3)
\psline(1,4)(5,4)\psline(2,1)(2,4)
\psline(3,2)(3,5)\psline(4,1)(4,5)
\end{pspicture} &
\begin{pspicture}(0.5,0.5)(6,6)
\psline(0.6,0.6)(0.6,5.4)(5.4,5.4)(5.4,0.6)(0.6,0.6)
\psline(0.6,2)(5.4,2)\psline(2,3)(4,3)
\psline(0.6,4)(5.4,4)\psline(2,0.6)(2,4)
\psline(3,2)(3,5.4)\psline(4,0.6)(4,5.4)
\psline(0.6,0.8)(5.4,0.8)\psline(0.6,1)(5.4,1)
\psline(0.6,5)(5.4,5)\psline(0.6,5.2)(5.4,5.2)
\psline(0.8,0.6)(0.8,5.4)\psline(1,0.6)(1,5.4)
\psline(5,0.6)(5,5.4)\psline(5.2,0.6)(5.2,5.4)
\end{pspicture} \\
$\mathscr{T}$  & $\mathscr{T}^\varepsilon$
\end{tabular}
\caption{A T-mesh $\mathscr{T}$ and its extended T-mesh $\mathscr{T}^{\varepsilon}$ for $\mathbf{S}^2(\mathscr{T})$.\label{extend}}
\end{center}
\end{figure}
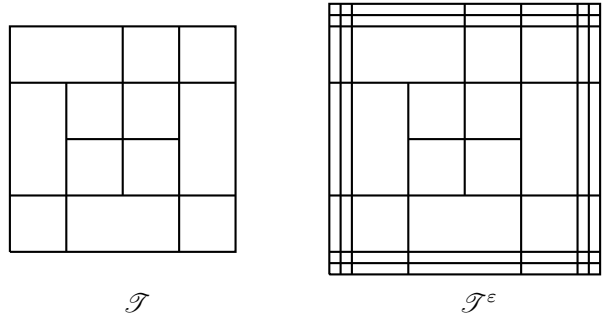

A spline space over a given T-mesh $\mathscr{T}$ with homogeneous boundary conditions (HBC for short) is defined by~\cite{Deng13}
$$
\overline{\mathbf{S}}(m,n,\alpha,\beta,\mathscr{T}):=\{f(x,y)\in
C^{\alpha,\beta}(\mathbb{R}^2): f(x,y)|_{\phi}\in
\mathbb{P}_{mn},\forall \phi\in\mathscr{F},\ \mathrm{ and }\
f|_{\mathbb{R}^2\setminus\Omega}\equiv0\},
$$
where $\mathbb{P}_{mn},\mathscr{F},C^{\alpha,\beta}$ are defined as before. One important observation in \cite{Deng13} is that the
two spline spaces $\mathbf{S}(m,n,m-1,n-1,\mathscr{T})$ and $\overline{\mathbf{S}}(m,n,m-1,n-1,\mathscr{T}^\varepsilon)$ are closely related.

\begin{theorem}[\cite{Deng13}]
 \label{extended}
 Given a T-mesh $\mathscr{T}$ , let $\mathscr{T}^{\varepsilon}$ be the extended T-mesh associated
with $\mathbf{S}(m,n,m-1,n-1,\mathscr{T})$. Then
\begin{align*}
\mathbf{S}(m,n,m-1,n-1,\mathscr{T})&=
                       \overline{\mathbf{S}}(m,n,m-1,n-1,
                       \mathscr{T}^{\varepsilon})|_{\Omega},\\
\dim \mathbf{S}(m,n,m-1,n-1,\mathscr{T})&=
                       \dim
                       \overline{\mathbf{S}}(m,n,m-1,n-1,\mathscr{T}^{\varepsilon}).
\end{align*}
\end{theorem}

In the T-mesh $\mathscr{T}$ of $\overline{\mathbf{S}}^d(\mathscr{T})$, for any l-edge $l$, the co-factors of all of the vertices on $l$
satisfy the linear system $\mathscr{S}_l=0$ as Equation \eqref{system}, which means we have the similar conformality vector spaces $W[\mathscr{T}]$
for all of the l-edges. HBC unifies all types of l-edges, which will bring convenience to the computation
of dimensions. From~\cite{Wu13}, we know that
$$
\overline{\mathbf{S}}^d(\mathscr{T})\cong W[\mathscr{T}].
$$

\begin{lemma}\label{tensor}
Given a T-mesh $\mathscr{T}$, we can construct a tensor-product mesh $\mathscr{T}_{\bigotimes}$ with all of the
boundary l-edges and some of the cross-cuts. See Figure~\ref{ten} for an example. Let $L$ be the set of all of the l-edges of $\mathscr{T}$
that are not contained in $\mathscr{T}_{\bigotimes}$. If $\mathscr{T}_{\bigotimes}$ has at least $d+1$
horizontal l-edges and $d+1$ vertical l-edges, then
$$
\dim \overline{\mathbf{S}}^d(\mathscr{T})=\dim \overline{\mathbf{S}}^d(\mathscr{T}_{\bigotimes})+\dim W[L].
$$
\end{lemma}

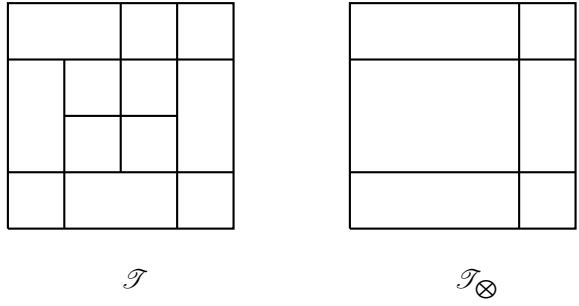
\begin{figure}[!htb]
\begin{center}
\psset{unit=0.75cm,linewidth=0.8pt}
\begin{tabular}{cc}
\begin{pspicture}(0.5,0.5)(6,6)
\psline(1,1)(1,5)(5,5)(5,1)(1,1) \psline(1,2)(5,2)\psline(2,3)(4,3)
\psline(1,4)(5,4)\psline(2,1)(2,4)
\psline(3,2)(3,5)\psline(4,1)(4,5)
\end{pspicture} &
\begin{pspicture}(0.5,0.5)(6,6)
\psline(1,1)(1,5)(5,5)(5,1)(1,1) \psline(1,2)(5,2)
\psline(1,4)(5,4)
\psline(4,1)(4,5)
\end{pspicture} \\
$\mathscr{T}$  & $\mathscr{T}_{\bigotimes}$
\end{tabular}
\caption{An example T-mesh for Lemma~\ref{tensor}. \label{ten}}
\end{center}
\end{figure}

\begin{proof}
Suppose $\mathscr{T}_{\bigotimes}$ is an $m\times n$ tensor-product mesh. Then $\dim \overline{\mathbf{S}}^d(\mathscr{T}_{\bigotimes})=(m-d-1)(n-d-1)$.
Now, we extend all of the interior l-edges in $L$ to the boundary, where the new l-edges set is denoted as $L_{1}$, and we create a new $m_{1}\times n_{1}$
tensor-product mesh $\mathscr{T}_{1}$. $\dim \overline{\mathbf{S}}^d(\mathscr{T}_{1})=(m_{1}-d-1)(n_{1}-d-1)$.

Because $\mathscr{T}_{\bigotimes}$ has at least $d+1$ horizontal l-edges and $d+1$ vertical l-edges, it is easy
to verify that $L_1$ has a reasonable order and we have $\dim W[L_{1}]=m_{1}n_{1}-mn-(d+1)((m_{1}-m)+(n_{1}-n))$, which means that
$$\dim \overline{\mathbf{S}}^d(\mathscr{T}_{1})=\dim \overline{\mathbf{S}}^d(\mathscr{T}_{\bigotimes})+\dim W[L_{1}].$$
Thus, the lemma is directly derived from Lemma~\ref{surject2}.
\end{proof}

\section{Dimensions of $\mathbf{S}^2(\mathscr{T})$ over Hierarchical T-Meshes\label{2211}}

In \cite{Deng13}, the dimension formula of $\mathbf{S}^2(\mathscr{T})$ over Hierarchical T-Meshes is provided, where the analysis is somewhat complex
and the main strategy is to construct a mapping from $\mathbf{S}^2(\mathscr{T})$ to $\mathbf{S}^1(\mathscr{T})$ and
decompose $\mathbf{S}^2(\mathscr{T})$ into the direct sum of some subspaces.
In this section, we provide a concise proof for the dimension formula of $\mathbf{S}^2(\mathscr{T})$
over a hierarchical T-mesh $\mathscr{T}$.

First, we define some notations. Suppose $\lev(\mathscr{T})=n$. We use
$V$, $V^+$, and $E$ to denote the numbers of the vertices, the crossing-vertices, and the interior l-edges of $\mathscr{T}$,
respectively. In addition, we use $V_i$, $V_i^+$, and $E_i$ to represent the numbers of the vertices,
the crossing-vertices, and the interior l-edges in $T_i$, respectively. Here, $T_i$ is the set of all of the level $i$ l-edges and the level $i$ vertices
of $\mathscr{T}$ (see Section \ref{hierarchy}).
It follows that
$$V=\sum_{i=0}^nV_i, \quad V^+=\sum_{i=0}^nV_i^+, \quad E=\sum_{i=0}^nE_i.$$
Then, the key procedure of the proof consists of the following components.
\begin{enumerate}
  \item Compute $\dim W[T_i],i=1,\ldots,n$;
  \item Prove that $\dim \overline{\mathbf{S}}^2(\mathscr{T})=\dim \overline{\mathbf{S}}^2(\mathscr{T}^0)+\dim W[T_1]+\dim W[T_2]+\cdots +\dim W[T_n]$.
\end{enumerate}

\begin{lemma} \label{t0_2}
If $\mathscr{T}^0$ has at least $3$ horizontal l-edges and $3$ vertical l-edges,
$\dim \overline{\mathbf{S}}^2(\mathscr{T}^0)=V_0^+-E_0+1$.
\end{lemma}
\begin{proof}Suppose $\mathscr{T}^0$ is an $m\times n$ tensor-product mesh. It follows that
$V_{0}^{+} = (m-2)(n-2)$, $E_{0} = m+n-4$ and
$\dim \overline{\mathbf{S}}^{2}(\mathscr{T}^0)=(m-3)(n-3) = V_0^+-E_0+1$.
\end{proof}

\begin{definition}
For a hierarchical T-mesh $\mathscr{T}$, $T_i$ $(i\geqslant 1)$ can be divided into
several parts, $T_{i_{1}},T_{i_{2}},\ldots,T_{i_{m_i}}$, such that the l-edges of $T_{i_{r}}$ do not
intersect with the l-edges of $T_{i_{s}}$, where $r,s\in \{1,2,\ldots,m_i\}$ and $r\neq s$.
$T_{i_k}$ is called a \textbf{connected component} of $T_i$, $k\in \{1,2,\ldots,m_i\}$.
\end{definition}

For the hierarchical T-mesh in Figure \ref{HTmesh}, $T_1$ has 1 connected component, and $T_2$ has 3 connected components.

\begin{lemma}\label{lemma:1}In a hierarchical T-mesh $\mathscr{T}$, suppose $T_{i_{1}},T_{i_{2}},\ldots,T_{i_{m_i}}$ are all of the connected components of $T_i$ $(i\geqslant 1)$,
then we have
\begin{equation}\label{con com}
\dim W[T_i]=\sum_{k=1}^{m_i}\dim W[T_{i_k}].
\end{equation}
\end{lemma}
\begin{proof}This can be directly derived from Lemma~\ref{2.3}.
\end{proof}

\begin{lemma}\label{dimT_i}
Suppose $T_i(i\geqslant 1)$ has $\delta_{1i}$ connected components that divide only one cell of $\mathscr{T}^{i-1}$. Then,
$ \dim W[T_i]=V_i^{+}-E_i+\delta_{1i}$.
\end{lemma}
\begin{proof}
According to Lemma~\ref{lemma:1}, we can only consider $\dim W[T_{i_{1}}]$. Suppose $T_{i_{1}}$
has $V_{i_{1}}$ vertices, $V_{i_{1}}^+$ crossing-vertices and $E_{i_{1}}$ l-edges.
Then $V_{i_{1}}=V_{i_{1}}^{+} + 2E_{i_{1}}$. Let $t_1,t_2,\ldots, t_k,t_{k+1},\ldots,t_r$ be all of the l-edges in $T_{i_{1}}$, where
$N(t_m)=1,m\in \{1,2,\ldots,k\}$ and $N(t_n)\geqslant 2,n\in \{k+1,\ldots,r\}$.

If $t_1,t_2,\ldots,t_k$ do not intersect with each other, $t_1,t_2,\ldots,t_k,\ldots,t_r$ form a reasonable
order of $T_{i_{1}}$. Actually, for any $j$, if $j \leqslant k$, the number of the vertices on $t_{j}$ but not on $t_{l}, l < j$,
is exactly 3 and for $j > k$, because $t_{j}$ crosses at least two cells in the upper level, the number of the vertices on $t_{j}$ but not on $t_{l}, l < j$,
is at least 3. Hence, we have $\dim W[T_{i_{1}}]=V_{i_{1}}-3E_{i_{1}}=V_{i_{1}}^{+} - E_{i_{1}}$.

Otherwise, $T_{i_{1}}$ only divides one cell of $\mathscr{T}^{i-1}$, and it is easy to determine that $\dim W[T_{i_{1}}]=0=V_{i_{1}}^{+}-E_{i_{1}}+1$. Therefore,
$$\dim W[T_{i_{1}}]= \begin{cases}
V_{i_{1}}^+-E_{i_{1}}+1,  & T_{i_{1}} \mbox{only divides one cell of} \mathscr{T}^{i-1}, \\
V_{i_{1}}^+-E_{i_{1}}, & \mbox{otherwise}.
\end{cases}
$$
Suppose there are $\delta_{1i}$ connected components of $T_i$ that divide only one cell of $\mathscr{T}^{i-1}$. Then,
$$
\dim W[T_i]=V_i^+-E_i+\delta_{1i}.
$$
\end{proof}

Let $T=T_1^o\cup T_2^o\cup\cdots \cup T_n^o$, $T'=T_2^o\cup\cdots \cup T_n^o$, where $T_i^o$ has the same meaning as in Section \ref{hierarchy}. For the second step of the proof, we only have to prove that $\dim \overline{\mathbf{S}}^2(\mathscr{T}^\varepsilon)=
\dim \overline{\mathbf{S}}^2(\mathscr{T}^0)+\dim W[T]$ and
$\dim W[T]=\dim W[T_1]+\dim W[T']$. First, we have the following lemma.

\begin{lemma}\label{+}
For $\overline{\mathbf{S}}^2(\mathscr{T})$, with the definition above, we have
$$
\dim W[T]=\sum_{i=1}^{n}\dim W[T_i].
$$
\end{lemma}
\begin{proof} We only need to prove that
$\dim W[T]=\dim W[T_1]+\dim W[T']$. According to Lemma \ref{dimT_i}, we suppose $T_1$ has only two connected components
$\overline{T_1}$ and $\widetilde{T_1}$. $\overline{T_1}$ divides only one cell of $\mathscr{T}^0$, while $\widetilde{T_1}$ divides more. Correspondingly, $T$ also has two connected components; say they are $\overline{T}$ and $\widetilde{T}$. $T'$ is divided into two parts; say they are $\overline{T'}$ and $\widetilde{T'}$, where $\overline{T'}\subseteq \overline{T}$, $\widetilde{T'}\subseteq \widetilde{T}$.

Because $\dim W[\widetilde{T_1}]=\widetilde{V}_1^+-\widetilde{E}_1$, where $\widetilde{V}_1^+$ and $\widetilde{E}_1$ are the numbers of the crossing-vertices
and the l-edges of $\widetilde{T_1}$, respectively, according to Lemma \ref{surject1}, we have $\dim W[\widetilde{T}]=\dim W[\widetilde{T_1}]+\dim W[\widetilde{T'}]$.
For $\overline{T_1}$, we consider the mesh $\mathscr{T}'$ generated by $\overline{T}$ and the edges of the cell of $\mathscr{T}^0$ divided by $\overline{T_1}$. Then, we have
\begin{align*}
\dim W[\overline{T'}] & =  \dim \overline{\mathbf{S}}^2(\mathscr{T}')
                       \leqslant  \dim W[\overline{T}]
                       \leqslant  \dim W[\overline{T'}]+\dim W[\overline{T_1}]
                    =  \dim W[\overline{T'}].
\end{align*}
That is, $\dim W[\overline{T}]=\dim W[\overline{T'}]=\dim W[\overline{T'}]+\dim W[\overline{T_1}]$, which proves the lemma.
\end{proof}

Now, we can give a new proof for the dimension formula of $\overline{\mathbf{S}}^2(\mathscr{T})$ provided in \cite{Deng13}.
\begin{theorem}[\cite{Deng13}]\label{hbc2}
Suppose $\mathscr{T}$ is a hierarchical T-mesh with $V^+$ crossing-vertices, $E$ interior l-edges and
$\delta_{1i}$ connected components in $T_i(i\geqslant 1)$ that divide only one cell of $\mathscr{T}^{i-1}$. Denote
$\delta_1=\sum_{i=1}^n\delta_{1i}$. If $\mathscr{T}^0$ has at least $3$ horizontal l-edges and $3$ vertical l-edges,
then $$\dim \overline{\mathbf{S}}^2(\mathscr{T})=V^+-E+\delta_1+1.$$
\end{theorem}
\begin{proof}According to Lemma~\ref{t0_2}, Lemma~\ref{dimT_i}, Lemma~\ref{tensor} and Lemma~\ref{+}, we have
\begin{align*}
\dim \overline{\mathbf{S}}^2(\mathscr{T}) &= \dim \overline{\mathbf{S}}^2(\mathscr{T}^0)+\dim W[T]  \\
 &= V_0^+-E_0+1 + \sum_{i=1}^{n} \dim W[T_{i}]  \\
 &=V_0^+-E_0+1 + \sum_{i=1}^{n} (V_i^{+}-E_i+\delta_{1i})  \\
 &= V^+-E+\delta_1+1.
\end{align*}
\end{proof}

\begin{theorem} \label{nonhbc2}
Suppose $\mathscr{T}$ is a hierarchical T-mesh with $V^+$ crossing-vertices, $V^b$ boundary vertices, $E$ interior l-edges and
$\delta_{1i}$ connected components in $T_i(i\geqslant 1)$ that divide only one interior cell of $\mathscr{T}^{i-1}$. Denote
$\delta_1=\sum_{i=1}^n\delta_{1i}$. Then $$\mathbf{S}^2(\mathscr{T})=2V^b+V^+-E+\delta_1+1.$$
\end{theorem}
\begin{proof}
According to Theorem \ref{extended}, we have only to compute $\dim \overline{\mathbf{S}}^2(\mathscr{T}^\varepsilon)$. It is easy
to check that $\mathscr{T}^{\varepsilon}$ has $2V^b+V^++8$ crossing-vertices, $E+8$ interior l-edges and $\delta_{1i}$
connected components in $T^\varepsilon_i(i\geqslant 1)$ that divide only one cell of $\mathscr{T}^{\varepsilon^{i-1}}$.

We should notice that $\mathscr{T}^\varepsilon$ is not a hierarchical T-mesh of $2\times 2$ division if any boundary cell of $\mathscr{T}$ is divided.
However, after extending, every l-edge of $\mathscr{T}^{\varepsilon^i}(i\geqslant 1)$ has more vertices than the original l-edge. Therefore, Theorem
\ref{hbc2} holds for $\mathscr{T}^\varepsilon$. Using the conclusion of Theorem \ref{hbc2} on $\mathscr{T}^\varepsilon$, we complete the proof.
\end{proof}

\section{Dimensions of $\mathbf{S}^3(\mathscr{T})$ over Hierarchical T-Meshes}\label{3322}

The difficulty in generalizing the above method to $\mathbf{S}^3(\mathscr{T})$ over a hierarchical T-mesh
is the fact that
the equation $\dim W[T]=\dim W[T_1]+\dim W[T']$ is not always true. For example, as indicated in Figure \ref{nolevel},
it is easy to compute that $\dim W[T_1]=0$, $\dim W[T_2]=6$, but $\dim W[T]\leqslant 5$. This tells us that we can not compute the
dimension of $\mathbf{S}^3(\mathscr{T})$ level by level if we consider general hierarchical T-meshes of $2\times2$  division. Thus, we
consider hierarchical T-meshes with some restrictions. In Section \ref{N(t)}, we consider hierarchical T-meshes with $N(t)\geqslant 2$ for any l-edge $t$ of $T$, and in Section \ref{3*3}, we consider hierarchical T-meshes of $3\times 3$ division. The symbols in this section have the same meanings as in Section \ref{2211}.
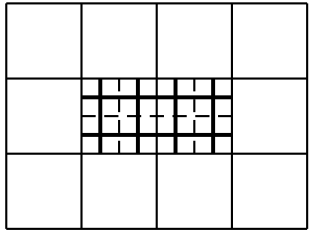
\begin{figure}[!ht]
\begin{center}
\begin{pspicture}(0,0)(4,4)
\psline(0,0)(0,3)\psline(1,0)(1,3)\psline(2,0)(2,3)\psline(3,0)(3,3)\psline(4,0)(4,3)\psline(0,0)(4,0)\psline(0,1)(4,1)\psline(0,2)(4,2)\psline(0,3)(4,3)
\psline[linewidth=1.5pt](1,1.25)(3,1.25)\psline[linestyle=dashed](1,1.5)(3,1.5)\psline[linewidth=1.5pt](1,1.75)(3,1.75)
\psline[linewidth=1.5pt](1.25,1)(1.25,2)\psline[linestyle=dashed](1.5,1)(1.5,2)\psline[linewidth=1.5pt](1.75,1)(1.75,2)
\psline[linewidth=1.5pt](2.25,1)(2.25,2)\psline[linestyle=dashed](2.5,1)(2.5,2)\psline[linewidth=1.5pt](2.75,1)(2.75,2)
\end{pspicture}
\caption{An example of a T-mesh. \label{nolevel}}
\end{center}
\end{figure}

\subsection{Hierarchical T-meshes with $N(t)\geqslant 2$}\label{N(t)}
\begin{lemma}\label{t0_3}
If $\mathscr{T}^0$ has at least $4$ horizontal l-edges and $4$ vertical l-edges,
then $\dim \overline{\mathbf{S}}^3(\mathscr{T}^0)=V_0^+ -2E_0+4$.
\end{lemma}
\begin{proof}Suppose $\mathscr{T}^0$ is an $m\times n$ tensor-product mesh. Then, $V_{0}^{+} = (m-2)(n-2)$, $E_{0} = m+n-4$, and
$\dim \overline{\mathbf{S}}^{3}(\mathscr{T}^0)=(m-4)(n-4) = V_0^+-2E_0+4$.
\end{proof}

\begin{lemma}\label{dimT_i3}
Suppose $T_i(i\geqslant 1)$ has $\delta_{4i}$ connected components that only divide the $2\times2$ neighbor cells in $\mathscr{T}^{i-1}$
(see Figure~\ref{well}.a). Then,
$ \dim W[T_i]=V_i^{+}-2E_i+\delta_{4i}$.
\end{lemma}
\begin{proof}
Similar to the discussion in the last section, we only consider $\dim W[T_{i_{1}}]$. Suppose $T_{i_{1}}$
has $V_{i_{1}}$ vertices, $V_{i_{1}}^+$ crossing-vertices and $E_{i_{1}}$ l-edges.
Then $V_{i_{1}}=V_{i_{1}}^{+} + 2E_{i_{1}}$. We also assume $t_1,t_2,\ldots,t_k,\ldots,t_r$ are all of the l-edges of $T_{i1}$,
where $N(t_i)=2$ when $i\leqslant k$; $N(t_i)>2$ when $i\geqslant k+1$.
For $i=1,2\ldots,k$, if there is an order, say $t_{j_1},t_{j_2},\ldots,t_{j_k}$, such that $t_{j_{p}}$ only intersects
with at most one l-edge of $t_{j_q}$, $q<p$, then $t_{j_1},t_{j_2},\ldots,t_{j_k}, t_{k+1}, \dots, t_{r}$ form a reasonable order. Hence,
$\dim W[T_{i_1}]=V_{i_1}-4E_{i_1}=V_{i_1}^+-2E_{i_1}$.

Otherwise, $T_{i_1}$ only divides the $2\times2$ neighbor cells in $\mathscr{T}^{i-1}$.
Then, we can obtain $\dim W[T_{i_1}] = 1$ according to~\cite{Wu13}. In this case, $\dim W[T_{i_1}]=1=V^+_{i_1}-2E_{i_1}+1$. Thus, if
$\delta_{4i}$ is the number of the connected components that only divide the $2\times2$ neighbor cells in $\mathscr{T}^{i-1}$, then
$ \dim W[T_i]=V_i^{+}-2E_i+\delta_{4i}$.
\end{proof}

\begin{figure}[!ht]
\begin{center}
\begin{tabular}{c@{\hspace*{1.5cm}}c}
\begin{pspicture}(0,0)(3,3)
\psline(0,0.5)(3,0.5)\psline[linestyle=dashed](0.5,1)(2.5,1)\psline(0,1.5)(3,1.5)\psline[linestyle=dashed](0.5,2)(2.5,2)\psline(0,2.5)(3,2.5)
\psline(0.5,0)(0.5,3)\psline[linestyle=dashed](1,0.5)(1,2.5)\psline(1.5,0)(1.5,3)\psline[linestyle=dashed](2,0.5)(2,2.5)\psline(2.5,0)(2.5,3)
\psline{->}(0.25,0.75)(0.5,0.75)\psline{->}(0.25,2.25)(1,2.25)
\rput[r](0.2,0.75){\small edge of level $i-1$}\rput[r](0.2,2.25){\small edge of level $i$}
\end{pspicture}  &
\begin{pspicture}(0,0)(3,3)
\psline(0.5,0.5)(2.5,0.5)\psline[linestyle=dashed](0.5,1)(2.5,1)\psline(0.5,1.5)(2.5,1.5)\psline[linestyle=dashed](0.5,2)(2.5,2)\psline(0.5,2.5)(2.5,2.5)
\psline(0.5,0.5)(0.5,2.5)\psline[linestyle=dashed](1,0.5)(1,2.5)\psline(1.5,0.5)(1.5,2.5)\psline[linestyle=dashed](2,0.5)(2,2.5)\psline(2.5,0.5)(2.5,2.5)
\end{pspicture} \\
a  & b
\end{tabular}
\caption{The special connected component $T_i$ (a) and the new mesh $\mathscr{T}'$ (b) \label{well}}
\end{center}
\end{figure}
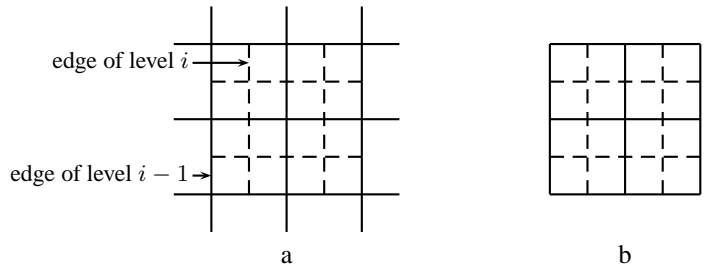

\begin{lemma}\label{++}
For $\overline{\mathbf{S}}^3(\mathscr{T})$, let $T=T_1^o\cup T_2^o\cup\cdots \cup T_n^o$. If $N(t) \geqslant 2$ for any l-edge $t$, we have
$$
\dim W[T]=\sum_{i=1}^{n}\dim W[T_i].
$$
\end{lemma}
\begin{proof}
Let $T'=T_2^o\cup\cdots \cup T_n^o$. Then, we have only to prove that
$\dim W[T]=\dim W[T_1]+\dim W[T']$. According to Lemma \ref{dimT_i3}, we suppose $T_1$ has only two connected components
$\overline{T_1}$ and $\widetilde{T_1}$. $\overline{T_1}$ only divides the $2\times2$ neighbor cells in $\mathscr{T}^0$,
while $\widetilde{T_1}$ does not. Correspondingly, $T$ also has two connected components; say they are $\overline{T}$ and
$\widetilde{T}$. $T'$ is divided into two parts; say they are $\overline{T'}$ and $\widetilde{T'}$, where
$\overline{T'}\subseteq \overline{T}$, $\widetilde{T'}\subseteq \widetilde{T}$.

Because $\dim W[\widetilde{T_1}]=\widetilde{V}_1^+-2\widetilde{E}_1$, where $\widetilde{V}_1^+,\widetilde{E}_1$ are the numbers of the
crossing-vertices
and the l-edges of $\widetilde{T_1}$, respectively; according to Lemma \ref{surject1}, we have $\dim W[\widetilde{T}]=\dim W[\widetilde{T_1}]+\dim W[\widetilde{T'}]$.
For $\overline{T_1}$, we consider the mesh $\mathscr{T}'$(Figure \ref{well}.b is the mesh excluding $\overline{T'}$) generated by $\overline{T}$ and the edges of the cells of $\mathscr{T}^0$ divided by $\overline{T_1}$. Then, we have
\begin{align*}
1+\dim W[\overline{T'}] =  \dim \overline{\mathbf{S}}^3(\mathscr{T}')
                       \leqslant  \dim W[\overline{T}]
                       \leqslant  \dim W[\overline{T'}]+\dim W[\overline{T_1}]
                       =  \dim W[\overline{T'}]+1.
\end{align*}
That is, $\dim W[\overline{T}]=\dim W[\overline{T'}]+1=\dim W[\overline{T'}]+\dim W[\overline{T_1}]$, which proves the lemma.
\end{proof}

\begin{theorem}
\label{thm6.4}
Suppose $\mathscr{T}$ is a hierarchical T-mesh with $V^+$ crossing-vertices, $E$ interior l-edges, and any l-edge $t$ of $T$ satisfies $N(t)\geqslant 2$. There are $\delta_{4i}$ connected components in $T_i$ that only divide the $2\times2$ neighbor cells in $\mathscr{T}^{i-1}$. Let $\delta_4=\sum_{i=1}^n\delta_{4i}$. If $\mathscr{T}^0$ has at least $4$ horizontal l-edges and $4$ vertical l-edges, then
$$
\dim \overline{\mathbf{S}}^3(\mathscr{T})= V^+-2E+4+\delta_4.
$$
\end{theorem}
\begin{proof}According to Lemma~\ref{t0_3}, Lemma~\ref{dimT_i3}, Lemma~\ref{tensor} and Lemma~\ref{++}, we have
\begin{align*}
\dim \overline{\mathbf{S}}^3(\mathscr{T}) &= \dim \overline{\mathbf{S}}^3(\mathscr{T}^0)+\dim W[T]  \\
 &= V_0^+-2E_0+4 + \sum_{i=1}^{n} \dim W[T_{i}]  \\
 &=V_0^+-2E_0+4+ \sum_{i=1}^{n} (V_i^{+}-2E_i+\delta_{4i})  \\
 &= V^+-2E+\delta_4+4.
\end{align*}
\end{proof}

Similar to Theorem \ref{nonhbc2}, we have the following theorem.
\begin{theorem}Suppose $\mathscr{T}$ is a hierarchical T-mesh with $V^+$ crossing-vertices, $V^b$ boundary vertices,
$E$ interior l-edges, and any l-edge $t$ of $T$ satisfies $N(t)\geqslant 2$. There are $\delta_{4i}$ connected components
in $T_i$ that only divide the $2\times2$ interior neighbor cells in $\mathscr{T}^{i-1}$. Let $\delta_4=\sum_{i=1}^n\delta_{4i}$. Then
$$
\dim {\mathbf{S}}^3(\mathscr{T})= 3V^b+V^+-2E+4+\delta_4.
$$
\end{theorem}

Now, we give an example for the dimension computation.

\begin{example}
In Figure \ref{examplefordim1}, there are 20 crossing-vertices (labeled with ``$\bullet$"),
16 boundary vertices (labeled with ``$\circ$"), and 12 interior l-edges (labeled with dashed lines).
$\delta_{42}=1$, $\delta_{41}=0$. Hence,
$\dim {\mathbf{S}}^3(\mathscr{T})=3\times 16+20-2\times 12+4+1=49$.
\end{example}

\begin{figure}[!ht]
\begin{center}
\psset{unit=0.75cm,linewidth=0.8pt}
\begin{pspicture}(0,0)(3,3)
\psline(0,0)(3,0)(3,3)(0,3)(0,0)
\psset{linestyle=dashed}
\psline(0,1)(3,1)\psline(0,2)(3,2)\psline(1,0)(1,3)\psline(2,0)(2,3)
\psline(1,1.5)(3,1.5)\psline(1,2.5)(3,2.5)\psline(1.5,1)(1.5,3)\psline(2.5,1)(2.5,3)
\psline(1.25,1)(1.25,2)\psline(1.75,1)(1.75,2)\psline(1,1.25)(2,1.25)\psline(1,1.75)(2,1.75)
\psdots[dotscale=0.7](1,1)(2,1)(1.25,1.25)(1.5,1.25)(1.75,1.25)(1.25,1.5)(1.5,1.5)(1.75,1.5)(2,1.5)(2.5,1.5)
(1.25,1.75)(1.5,1.75)(1.75,1.75)(1,2)(1.5,2)(2,2)(2.5,2)(1.5,2.5)(2,2.5)(2.5,2.5)
\psdots[dotscale=0.7,dotstyle=o](0,0)(1,0)(2,0)(3,0)(0,1)(3,1)(3,1.5)(0,2)(3,2)(3,2.5)(0,3)(1,3)(1.5,3)(2,3)(2.5,3)(3,3)
\end{pspicture}
\caption{A hierarchical T-mesh $\mathscr{T}$. \label{examplefordim1}}
\end{center}
\end{figure}
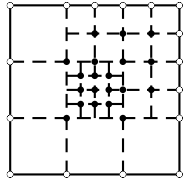
\subsubsection{The topological explanation}

In \cite{Deng13}, the following structure is introduced to propose a topological explanation to the dimension formula.
\begin{definition}[\cite{Deng13}]
Given a hierarchical T-mesh $\mathscr{T}$, we can construct a graph $\mathscr{G}$ by retaining the crossing-vertices and the edges with two end points that are crossing-vertices and removing the other vertices and the edges in $\mathscr{T}$. $\mathscr{G}$ is called the
crossing-vertex-relationship graph (CVR graph for short) of $\mathscr{T}$.
\end{definition}

\begin{figure}[!ht]
\begin{center}
\psset{unit=0.75cm,linewidth=0.8pt}
\scriptsize
\begin{tabular}{c@{\hspace{1.3cm}}c}
\begin{pspicture}(0,0)(5,4)
\psline(0,0)(5,0)(5,4)(0,4)(0,0)\psline(0,0.5)(2,0.5)\psline(0,1)(5,1)\psline(0,1.5)(3,1.5)\psline(0,2)(5,2)\psline(2,2.5)(5,2.5)\psline(0,3)(5,3)
\psline(3,3.5)(5,3.5)\psline(0.5,0)(0.5,2)\psline(1,0)(1,4)\psline(1.5,0)(1.5,2)\psline(2,0)(2,4)\psline(2.5,1)(2.5,3)\psline(3,0)(3,4)
\psline(3.5,2)(3.5,4)\psline(4,0)(4,4)\psline(4.5,2)(4.5,4)
\end{pspicture}  &
\begin{pspicture}(0,0)(5,4)
\psline(0.5,0.5)(1.5,0.5)\psline(0.5,1)(4,1)\psline(0.5,1.5)(2.5,1.5)\psline(1,2)(4,2)\psline(2.5,2.5)(4.5,2.5)\psline(1,3)(4.5,3)
\psline(3.5,3.5)(4.5,3.5)\psline(0.5,0.5)(0.5,1.5)\psline(1,0.5)(1,3)\psline(1.5,0.5)(1.5,1.5)\psline(2,1)(2,3)\psline(2.5,1.5)(2.5,2.5)\psline(3,1)(3,3)
\psline(3.5,2.5)(3.5,3.5)\psline(4,1)(4,3.5)\psline(4.5,2.5)(4.5,3.5)\psdots[dotscale=0.75](0.5,1)(0.5,1.5)(1,1.5)(1.5,1.5)(2,1.5)(2.5,1.5)\rput(0.3,1){$1$}
\rput(0.3,1.5){$2$}\rput(0.8,1.7){$3$}\rput(1.5,1.7){$T$}\rput(2.2,1.3){$+$}\rput(2.7,1.5){$L$}
\end{pspicture} \\
$\mathscr{T}$ & CVR graph of $\mathscr{T}$
\end{tabular}
\caption{A hierarchical T-mesh and its CVR graph. \label{cvr}}
\end{center}
\end{figure}
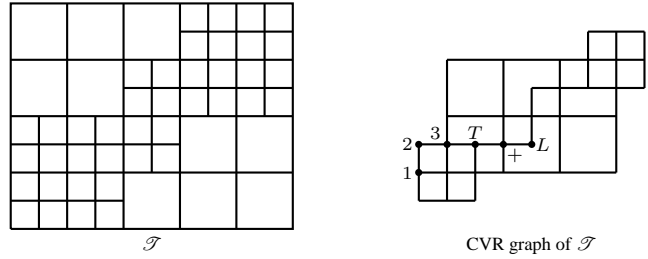

It should be noted that, when $N(t)\geqslant 2$ for any l-edge of $t$ in a hierarchical T-mesh $\mathscr{T}$, the CVR graph of $\mathscr{T}$ is connected.
See Figure~\ref{cvr} for an example. In addition, in \cite{Deng13}, the following conjecture is stated:
\begin{conjecture}
\label{conjecture}
Suppose $\mathscr{T}$ is a hierarchical T-mesh whose CVR graph is $\mathscr{G}$. When $m, n
\geqslant 2$, it follows that
$$\dim \overline{\mathbf{S}}(m, n, m-1, n-1, \mathscr{T}) = \dim
\overline{\mathbf{S}}(m-2, n-2,m-3, n-3, \mathscr{G}),$$
where $\overline{\mathbf{S}}(m,n,\alpha,\beta,\mathscr{G})$
is defined in a similar way to the spline space over a T-mesh.
\end{conjecture}
In \cite{Deng13}, the conjecture is proved to hold as $m=n=2$. Now, we will prove that it also holds
as $m=n=3$ over a hierarchical T-mesh specified in Theorem \ref{thm6.4}.

In $\mathscr{G}$, we can define boundary vertex, interior vertex, crossing-vertex, T-junction, edge, l-edge, and so on in the
same way as in $\mathscr{T}$. We use $V_{\mathscr{G}}$ to denote the number of the vertices of the CVR graph. Therefore,
$V_{\mathscr{G}}=V^+$. In $\mathscr{G}$, a vertex is the intersection of two l-edges. We distinguish the vertices of $\mathscr{G}$
into six types as follows:
\begin{enumerate}
\item Type 1: a boundary vertex of $\mathscr{G}$, which is an end point of an l-edge and an interior vertex of another l-edge,
\item Type 2: a boundary vertex of $\mathscr{G}$, which is the end points of two l-edges,
\item Type 3: a boundary vertex of $\mathscr{G}$, which is the interior vertices of two l-edges,
\item Type T: an interior vertex of $\mathscr{G}$, which is a  T-junction of $\mathscr{G}$,
\item Type $+$: an interior vertex of $\mathscr{G}$, which is a crossing-vertex of $\mathscr{G}$,
\item Type L: an interior vertex of $\mathscr{G}$, which is the end points of two l-edges.
\end{enumerate}
The six types of vertices have been marked in Figure \ref{cvr}. We use
$V_{\mathscr{G}}^1,V_{\mathscr{G}}^2,V_{\mathscr{G}}^3,V_{\mathscr{G}}^{\mathrm{T}},V_{\mathscr{G}}^{+},V_{\mathscr{G}}^{\mathrm{L}}$
to denote the numbers of the six types of vertices, respectively. 

\begin{lemma}\label{+4}
$V_{\mathscr{G}}^{2}=V_{\mathscr{G}}^{3}+4$.
\end{lemma}
\begin{proof}
The CVR graph $\mathscr{G}$ is connected.
We choose a vertex $v$ as the starting point on the boundary of $\mathscr{G}$, and run through the boundary of $\mathscr{G}$ in the clockwise direction.
When we meet a vertex of type $2$, the direction rotates $90^\circ$ clockwise; and when we meet a vertex of type $3$, the direction rotates
$90^\circ$ anti-clockwise. Finally, we return to $v$, and the direction rotates $360^\circ$ clockwise. Therefore,
$$
90V_{\mathscr{G}}^{2}-90V_{\mathscr{G}}^{3}=360.
$$
That is $V_{\mathscr{G}}^{2}=V_{\mathscr{G}}^{3}+4.$
\end{proof}

\begin{theorem}
Suppose $\mathscr{T}$ is a hierarchical T-mesh with $N(t)\geqslant 2$ for any l-edge $t$, and its CVR graph
is $\mathscr{G}$. There are $\delta_{4i}$ connected components as in Figure \ref{well}.a
in $T_i$. $\delta_4=\sum_{i=1}^n\delta_{4i}$.  Then
$$
\dim \overline{\mathbf{S}}^3(\mathscr{T})= V_{\mathscr{G}}^{+}-V_{\mathscr{G}}^{L}+\delta_4.
$$
\end{theorem}
\begin{proof}
The end points of every l-edge in $\mathscr{G}$ could be the types of $1,2,T,L$. Running through all of the l-edges of the CVR graph, we get $2E=V_{\mathscr{G}}^1+2V_{\mathscr{G}}^2+V_{\mathscr{G}}^T+2V_{\mathscr{G}}^L$. Then,
\begin{align*}
V^+-2E+4 &=V_{\mathscr{G}}-V_{\mathscr{G}}^1-2V_{\mathscr{G}}^2-V_{\mathscr{G}}^T-2V_{\mathscr{G}}^L  \\
         &=V_{\mathscr{G}}^{+}-V_{\mathscr{G}}^{L}+V_{\mathscr{G}}^{3}-V_{\mathscr{G}}^{2}+4.
\end{align*}
According to Lemma \ref{+4}, we know the theorem holds.
\end{proof}

Now, we prove Conjecture \ref{conjecture} holds for $\overline{\mathbf{S}}^3(\mathscr{T})$ over a hierarchical T-mesh with
$N(t)\geqslant 2$ for any l-edge $t$.
\begin{theorem}\label{cvr=dim}
Suppose $\mathscr{T}$ is a hierarchical T-mesh with $N(t)\geqslant 2$ for any l-edge $t$. If the CVR graph of $\mathscr{T}$ is $\mathscr{G}$, we have
$$ \dim \overline{\mathbf{S}}^3(\mathscr{T}) = \dim
\overline{\mathbf{S}}^1(\mathscr{G}).$$
\end{theorem}

\begin{proof}
$\mathscr{G}$ has a natural hierarchical structure, and we use the same notations in Section \ref{hierarchy}.

First, we have $\dim W[G_i]=\dim W[T_i]$. Suppose all of the l-edges of $T_i$ are $t_1,t_2,\ldots,t_m$. Every vertex of $t_j,j=1,2,\ldots,m$, excluding
the two end points,
is a crossing-vertex. Because $N(t_j)\geqslant 2$, $t_j$ has at least $5$ vertices. Therefore, $G_i$ also has $m$ l-edges $t'_1,t'_2,\ldots,
t'_m$, where $t'_j$ is derived from $t_j$ and has two less vertices than $t_j$. $t'_l$ intersects with $t'_k$ when $t_l$ intersects $t_k$,
where $k \neq l,k,l\in \{1,2,\ldots,m\}$. Using the method similar to the computation of $\dim W[T_i]$, we can get $\dim W[G_i]$ easily,
and we have $\dim W[G_i]=\dim W[T_i]$.

Second, we have $\dim W[G]=\sum_{i=1}^n \dim W[G_i]$, where $G=G_1^o\cup G_2^o\cup\cdots \cup G_n^o$. We can suppose that $n=2$. When $T_1$
only divides the $2\times 2$ neighbor cells in $\mathscr{T}^0$, according to Lemma \ref{surject2}, we should only prove this conclusion
when all of the cells of level $1$ are divided (see Figure \ref{cvrlevel}.a).
Figure \ref{cvrlevel}.b is the corresponding CVR graph, in which $G_2$ consists of the thick lines and $G$ consists of
the thick lines and the dashed lines. From the first part, we know $\dim W[G_1]=1,\dim W[G_2]=24$. To compute $\dim W[G]$, we consider the mesh $\mathscr{G}'$
of Figure \ref{cvrlevel}.b. It is easy to determine that $\dim \overline{\mathbf{S}}^1(\mathscr{G}')=25$. We have $\dim W[G]=\dim \overline{\mathbf{S}}^1(\mathscr{G}')=25$. Therefore, $\dim W[G]= \dim W[G_1]+\dim W[G_2]$. When $T_1$ divides more cells, $G_1$ has a reasonable order,
we also have $\dim W[G]= \dim W[G_1]+\dim W[G_2]$.

Third, we have $\dim \overline{\mathbf{S}}^1(\mathscr{G})=\dim \overline{\mathbf{S}}^1(\mathscr{G}^0)+\dim W[G]$. This is a generalization of Lemma \ref{tensor}.
If any boundary cell of $\mathscr{T}^0$ is divided, the bound of $\mathscr{G}$ may be not a rectangle. In this situation, we can
use the same method as in the second part to change $\mathscr{G}$ into a regular mesh. We omit the proof process here.

Combining the three parts, we have proved this theorem.
\end{proof}

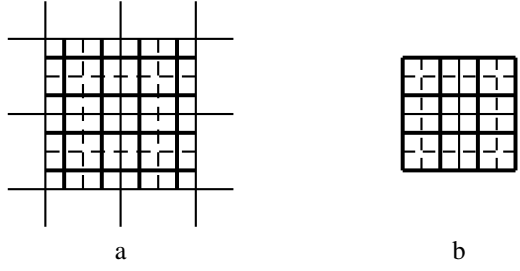
\begin{figure}[!ht]
\begin{center}
\begin{tabular}{c@{\hspace*{1.5cm}}c}
\begin{pspicture}(0,0)(3,3)
\psline(0,0.5)(3,0.5)\psline[linestyle=dashed](0.5,1)(2.5,1)\psline(0,1.5)(3,1.5)\psline[linestyle=dashed](0.5,2)(2.5,2)\psline(0,2.5)(3,2.5)
\psline(0.5,0)(0.5,3)\psline[linestyle=dashed](1,0.5)(1,2.5)\psline(1.5,0)(1.5,3)\psline[linestyle=dashed](2,0.5)(2,2.5)\psline(2.5,0)(2.5,3)
\psline[linewidth=1.5pt](0.5,0.75)(2.5,0.75)\psline[linewidth=1.5pt](0.5,1.25)(2.5,1.25)\psline[linewidth=1.5pt](0.5,1.75)(2.5,1.75)
\psline[linewidth=1.5pt](0.5,2.25)(2.5,2.25)
\psline[linewidth=1.5pt](0.75,0.5)(0.75,2.5)\psline[linewidth=1.5pt](1.25,0.5)(1.25,2.5)\psline[linewidth=1.5pt](1.75,0.5)(1.75,2.5)
\psline[linewidth=1.5pt](2.25,0.5)(2.25,2.5)
\end{pspicture}  &
\begin{pspicture}(0,0)(3,3)
\psline(0.75,1.5)(2.25,1.5)\psline(1.5,0.75)(1.5,2.25)
\psline[linestyle=dashed](0.75,1)(2.25,1)\psline[linestyle=dashed](0.75,2)(2.25,2)
\psline[linestyle=dashed](1,0.75)(1,2.25)\psline[linestyle=dashed](2,0.75)(2,2.25)
\psline[linewidth=1.5pt](0.75,0.75)(2.25,0.75)\psline[linewidth=1.5pt](0.75,1.25)(2.25,1.25)\psline[linewidth=1.5pt](0.75,1.75)(2.25,1.75)
\psline[linewidth=1.5pt](0.75,2.25)(2.25,2.25)
\psline[linewidth=1.5pt](0.75,0.75)(0.75,2.25)\psline[linewidth=1.5pt](1.25,0.75)(1.25,2.25)\psline[linewidth=1.5pt](1.75,0.75)(1.75,2.25)
\psline[linewidth=1.5pt](2.25,0.75)(2.25,2.25)
\end{pspicture} \\
a  & b
\end{tabular}
\caption{Figure for the proof of theorem \ref{cvr=dim}. \label{cvrlevel}}
\end{center}
\end{figure}

\subsection{Hierarchical T-meshes of $3\times 3$ division}\label{3*3}
In this section we consider hierarchical T-meshes of $3\times 3$ division. See Figure~\ref{9} for an example.

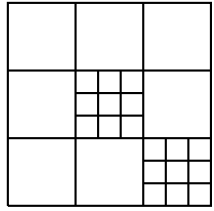
\begin{figure}[!ht]
\begin{center}
\begin{pspicture}(0,0)(3,3)
\psline(0,0)(0,2.7)(2.7,2.7)(2.7,0)(0,0)
\psline(0,0.9)(2.7,0.9)\psline(0,1.8)(2.7,1.8)\psline(0.9,0)(0.9,2.7)\psline(1.8,0)(1.8,2.7)
\psline(1.2,0.9)(1.2,1.8)\psline(1.5,0.9)(1.5,1.8)\psline(0.9,1.2)(1.8,1.2)\psline(0.9,1.5)(1.8,1.5)
\psline(2.1,0)(2.1,0.9)\psline(2.4,0)(2.4,0.9)\psline(1.8,0.3)(2.7,0.3)\psline(1.8,0.6)(2.7,0.6)
\end{pspicture}
\caption{A hierarchical T-mesh of $3\times 3$ division. \label{9}}
\end{center}
\end{figure}

Before discussing the dimension of $W[T_i]$, we first introduce some notations. Suppose $T_i$ only has one connected component,
and all of the l-edges are $t_{11},\ldots,t_{1r},t_{21},\ldots,t_{2s},$ $t_1,\ldots,t_m$, where
$N(t_{1j})=1,N(t_{2k})=2,N(t_l)\geqslant 3,j\in\{1,\ldots,r\},k\in\{1,\dots,s\},$ $l\in\{1,\ldots,m\}$. Then, $T_{i}$ can be
the following several types:
\begin{enumerate}
  \item $m > 0$,
  \item $m = 0, s > 0$. This case includes five possible sub-cases:
  \begin{enumerate}
    \item The connected component refines two cells in the $i-1$-th level, see Figure~\ref{n(t)}.a.
    \item The connected component refines three cells in the $i-1$-th level, see Figure~\ref{n(t)}.b.
    \item The connected component refines four cells in the $i-1$-th level which are the $2\times2$ neighbor cells, see Figure~\ref{n(t)}.c.
    \item The connected component refines four cells in the $i-1$-th level which are not the $2\times2$ neighbor cells, see Figure~\ref{n(t)}.d.
    \item The connected component refines more than four cells in the $i-1$-th level, see Figure~\ref{n(t)}.e.
  \end{enumerate}
  \item $m = s = 0$, see Figure~\ref{n(t)}.f.
\end{enumerate}

\begin{figure}[!ht]
\begin{center}
\begin{tabular}{ccc}
\begin{pspicture}(0,0)(2.4,1.5)
\psline(0,0.3)(2.4,0.3)\psline(0,1.2)(2.4,1.2)\psline(0.3,0)(0.3,1.5)\psline(1.2,0)(1.2,1.5)\psline(2.1,0)(2.1,1.5)
\psset{linestyle=dashed}
\psline(0.6,0.3)(0.6,1.2)\psline(0.9,0.3)(0.9,1.2)\psline(1.5,0.3)(1.5,1.2)\psline(1.8,0.3)(1.8,1.2)
\psline(0.3,0.6)(2.1,0.6)\psline(0.3,0.9)(2.1,0.9)
\end{pspicture} &
\begin{pspicture}(0,0)(2.4,2.4)
\psline(0,0.3)(2.4,0.3)\psline(0,1.2)(2.4,1.2)\psline(0,2.1)(2.4,2.1)\psline(0.3,0)(0.3,2.4)\psline(1.2,0)(1.2,2.4)\psline(2.1,0)(2.1,2.4)
\psset{linestyle=dashed}
\psline(0.6,0.3)(0.6,1.2)\psline(0.9,0.3)(0.9,1.2)\psline(1.5,0.3)(1.5,2.1)\psline(1.8,0.3)(1.8,2.1)
\psline(0.3,0.6)(2.1,0.6)\psline(0.3,0.9)(2.1,0.9)\psline(1.2,1.5)(2.1,1.5)\psline(1.2,1.8)(2.1,1.8)
\end{pspicture} &
\begin{pspicture}(0,0)(2.4,2.4)
\psline(0,0.3)(2.4,0.3)\psline(0,1.2)(2.4,1.2)\psline(0,2.1)(2.4,2.1)\psline(0.3,0)(0.3,2.4)\psline(1.2,0)(1.2,2.4)\psline(2.1,0)(2.1,2.4)
\psset{linestyle=dashed}
\psline(0.6,0.3)(0.6,2.1)\psline(0.9,0.3)(0.9,2.1)\psline(1.5,0.3)(1.5,2.1)\psline(1.8,0.3)(1.8,2.1)
\psline(0.3,0.6)(2.1,0.6)\psline(0.3,0.9)(2.1,0.9)\psline(0.3,1.5)(2.1,1.5)\psline(0.3,1.8)(2.1,1.8)
\end{pspicture} \\
a  & b & c \\
\begin{pspicture}(0,0)(3.3,2.4)
\psline(0,0.3)(3.3,0.3)\psline(0,1.2)(3.3,1.2)\psline(0,2.1)(3.3,2.1)\psline(0.3,0)(0.3,2.4)\psline(1.2,0)(1.2,2.4)\psline(2.1,0)(2.1,2.4)\psline(3,0)(3,2.4)
\psset{linestyle=dashed}
\psline(0.6,0.3)(0.6,1.2)\psline(0.9,0.3)(0.9,1.2)\psline(1.5,0.3)(1.5,2.1)\psline(1.8,0.3)(1.8,2.1)
\psline(0.3,0.6)(2.1,0.6)\psline(0.3,0.9)(2.1,0.9)\psline(1.2,1.5)(3,1.5)\psline(1.2,1.8)(3,1.8)
\psline(2.4,1.2)(2.4,2.1)\psline(2.7,1.2)(2.7,2.1)
\end{pspicture}  &
\begin{pspicture}(0,0)(3.3,3.3)
\psline(0,0.3)(3.3,0.3)\psline(0,1.2)(3.3,1.2)\psline(0,2.1)(3.3,2.1)\psline(0,3)(3.3,3)
\psline(0.3,0)(0.3,3.3)\psline(1.2,0)(1.2,3.3)\psline(2.1,0)(2.1,3.3)\psline(3,0)(3,3.3)
\psset{linestyle=dashed}
\psline(0.6,0.3)(0.6,1.2)\psline(0.9,0.3)(0.9,1.2)\psline(1.5,0.3)(1.5,2.1)\psline(1.8,0.3)(1.8,2.1)
\psline(0.3,0.6)(2.1,0.6)\psline(0.3,0.9)(2.1,0.9)\psline(1.2,1.5)(3,1.5)\psline(1.2,1.8)(3,1.8)
\psline(2.4,1.2)(2.4,3)\psline(2.7,1.2)(2.7,3)\psline(2.1,2.4)(3,2.4)\psline(2.1,2.7)(3,2.7)
\end{pspicture}  &
\begin{pspicture}(0,0)(1.5,1.5)
\psline(0,0.3)(1.5,0.3)\psline(0,1.2)(1.5,1.2)\psline(0.3,0)(0.3,1.5)\psline(1.2,0)(1.2,1.5)
\psset{linestyle=dashed}
\psline(0.6,0.3)(0.6,1.2)\psline(0.9,0.3)(0.9,1.2)
\psline(0.3,0.6)(1.2,0.6)\psline(0.3,0.9)(1.2,0.9)
\end{pspicture} \\
d & e & f
\end{tabular}
\caption{Special cases of $T_i$. \label{n(t)}}
\end{center}
\end{figure}
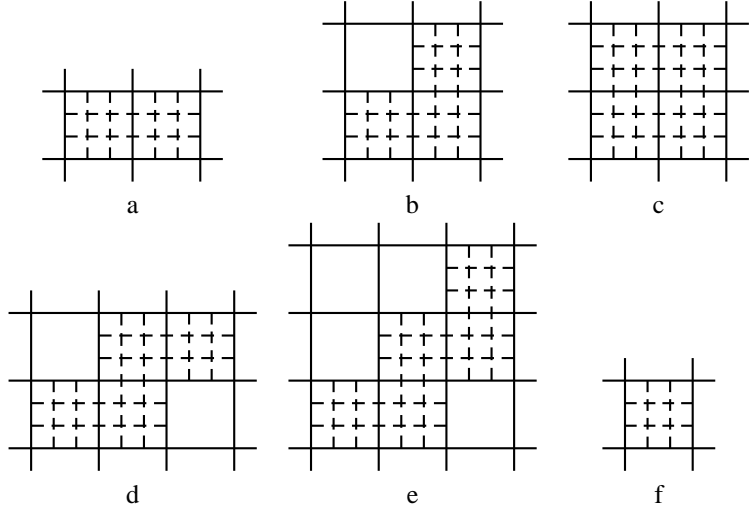

\begin{lemma}\label{dimT_i33}
Suppose $T_i(i\geqslant 1)$ has $\delta_{1i}$ connected components of case 3, $\delta_{2i}$ connected components of case 2.a,
$\delta_{3i}$ connected components of case 2.b and $\delta_{4i}$ connected components of case 2.c. Then $ \dim W[T_i]=V_i^{+}- 2E_i+ 4\delta_{1i} + 2\delta_{2i} + \delta_{3i}+\delta_{4i}$.
\end{lemma}
\begin{proof}
Suppose $T_i$ only has one connected component. We discuss every case of $T_i$ discussed above.
\begin{enumerate}
  \item $m > 0$:

  First, we should mention that we can remove all of $t_{11},\ldots,t_{1r}$ without changing the dimension. For the other l-edges,
  if $s = 0$, it is obvious that $t_1,\ldots,t_m$ have a reasonable order. If $s>0$, we define the order of the l-edges such that it
  is a reasonable order. First, select any l-edge as $t_{21}$. If there is an l-edge $t$ with $N(t)=2$ intersecting with $t_{21}$, we choose $t$ as $t_{22}$.
  If there is an l-edge $t'$ with $N(t')=2$ intersecting with $t_{22}$, we choose $t'$ as $t_{23}$. Repeat this process until all of
  the l-edges that cross two cells on the upper level have been sorted. In this case, $t_{11},\ldots,t_{1r},t_{21},\ldots,t_{2s},$ $t_1,\ldots,t_m$ is a reasonable order.

  \item $m = 0$, $s > 0$. This case includes five possible sub-cases:
  \begin{enumerate}
    \item The connected component refines two cells in the $i-1$-th level.

    It is evident that $\dim W[T_i]=0=V_i^+-2E_i+2$.

    \item The connected component refines three cells in the $i-1$-th level.

    It is evident that $\dim W[T_i]=1=V_i^+-2E_i+1$.

    \item The connected component refines four cells in the $i-1$-th level that are the $2\times2$ neighbor cells.

    It is evident that $\dim W[T_i]=9=V_i^+-2E_i+1$.

    \item The connected component refines four cells in the $i-1$-th level that are not the $2\times2$ neighbor cells.

    We delete $t_{1r}$ and get the mesh $S_i$ in Figure \ref{comp}.

    First, we prove that the projection mapping $\pi : W[S_i]\rightarrow W[t]$ is surjective. Because $t$ is an l-edge of order one, we only need to verify that the
    co-factors of the vertices on $t$ could be nonzero in $W[S_i]$. Let $W[t_1]=W[t_2]=0$. Then we change $S_i$ to $S'_i$ (see Figure~\ref{comp}) as with the mesh in case 2.b. Because $\dim W[S'_i]=1$, the co-factor of every vertex on $S'_i$ is nonzero. Thus, the mapping $\pi$ is surjective.
    We have $\dim W[S_i]=\dim W[S''_i]+1$, where $S''_i$ is the mesh $S_i\setminus t$ (See Figure~\ref{comp}).
    According to Equation \eqref{lowbound}, we have $\dim W[S''_i]\geqslant \dim W[S'_i]=1$. Therefore, $\dim W[S_i]\geqslant 2$.

    For the upper bound, we let $\gamma(v_0)=\gamma(v'_0)=0$, where $\gamma(v_0),\gamma(v'_0)$ are the co-factors of $v_0$ and $v'_0$, respectively. Then, $S_i$ is changed into $S'''_i$ (see Figure \ref{comp}).
    Suppose the lengths of the four edges of $t_1$ are $l$, $d$, $d$, $d$, respectively. We will prove that $\dim W[S'''_i]=0$.
    Let $\gamma(v_1)=x$; according to Equation \eqref{1-order}, we can obtain that
    $$\gamma(v_2)=ax, \quad  a=-\frac{d+l}{2d+l}, \quad -1<a<0.$$
    Similarly, we have $$\gamma(v_3)=abx,\quad \gamma(v_4)=abcx,\quad b\in \mathbb{R},-1<c<0.$$
    We can get $\gamma(v_1)=acx$ from $\gamma(v_4)$. Thus, we have $x=acx$. Because $0<ac<1$, we get $x=0$, which means that
    $\dim W[S'''_i]=0$. Therefore, $\{\gamma(v_0)$, $\gamma(v'_0)\}$ is the determining set of $W[S_i]$. Then, $\dim W[S_i]\leqslant 2$.

    Combining the former two parts, we have $\dim W[S_i]=2$; and it is easy to check that $\dim W[S_i]=V_i^+-2E_i$.

    \item The connected component refines more than four cells in the $i-1$-th level.\\
    This case is illustrated in Figure~\ref{n(t)}.e. Suppose it refines $p$ connected cells in the $i-1$-th level. As $m = 0$, so
    after removing all of $t_{11},\ldots,t_{1r}$, $T_i$ can be regarded as a mesh that is generated by adding two l-edges of order one to
    the mesh refining $p-1$ connected cells in the $i-1$-th level. For example, Figure~\ref{n(t)}.e is the mesh generated by adding two l-edges of order one
    to the mesh
    $S_i$ in Figure \ref{comp}.
    Continuing this process until the T-mesh follows the case 2.d, and according to Lemma~\ref{surject1}, we know that $\dim W[T_i]=V_i^+-2E_i$.
  \end{enumerate}
  \item $m = s = 0$. \\
  It is evident that $\dim W[T_i]=0=V_i^+-2E_i+4$.
\end{enumerate}
Combining the above discussion, we obtain this lemma.
\end{proof}

\begin{figure}[!ht]
\begin{center}
\begin{tabular}{c@{\hspace{1.3cm}}c}
\begin{pspicture}(0,0)(4.4,3.2)
\psline(0,0.4)(4.4,0.4)\psline(0,1.6)(4.4,1.6)\psline(0,2.8)(4.4,2.8)\psline(0.4,0)(0.4,3.2)\psline(1.6,0)(1.6,3.2)\psline(2.8,0)(2.8,3.2)\psline(4,0)(4,3.2)
\psline{->}(1.8,2.9)(1.8,2.4)\psline{->}(1,1.7)(1,1.2)\psline{->}(1,0.3)(1,0.8)
\psset{linestyle=dashed}
\psline(2,0.4)(2,2.8)\psline(2.4,0.4)(2.4,2.8)
\psline(0.4,0.8)(2.8,0.8)\psline(0.4,1.2)(2.8,1.2)\psline(1.6,2)(4,2)\psline(1.6,2.4)(4,2.4)
\psdots(2,1.2)(2,2)
\rput(1.8,3.05){$t$}\rput(1,1.85){$t_2$}\rput(1,0.15){$t_1$}\rput(1.82,1.02){$v_0$}\rput(1.82,1.82){$v'_0$}
\end{pspicture}  &
\begin{pspicture}(0,0)(4.4,3.2)
\psline(0,0.4)(4.4,0.4)\psline(0,1.6)(4.4,1.6)\psline(0,2.8)(4.4,2.8)\psline(0.4,0)(0.4,3.2)\psline(1.6,0)(1.6,3.2)\psline(2.8,0)(2.8,3.2)\psline(4,0)(4,3.2)
\psline{->}(1.8,2.9)(1.8,2.4)
\psset{linestyle=dashed}
\psline(2,0.4)(2,2.8)\psline(2.4,0.4)(2.4,2.8)
\psline(1.6,2)(4,2)\psline(1.6,2.4)(4,2.4)
\rput(1.8,3.05){$t$}
\end{pspicture} \\
$S_i$ & $S'_i$   \\
\begin{pspicture}(0,0)(4.4,3.2)
\psline(0,0.4)(4.4,0.4)\psline(0,1.6)(4.4,1.6)\psline(0,2.8)(4.4,2.8)\psline(0.4,0)(0.4,3.2)\psline(1.6,0)(1.6,3.2)\psline(2.8,0)(2.8,3.2)\psline(4,0)(4,3.2)
\psline{->}(1,1.7)(1,1.2)\psline{->}(1,0.3)(1,0.8)
\psset{linestyle=dashed}
\psline(2,0.4)(2,2.8)\psline(2.4,0.4)(2.4,2.8)
\psline(0.4,0.8)(2.8,0.8)\psline(0.4,1.2)(2.8,1.2)\psline(1.6,2)(4,2)
\psdots(2,1.2)(2,2)
\rput(1,1.85){$t_2$}\rput(1,0.15){$t_1$}\rput(1.82,1.02){$v_0$}\rput(1.82,1.82){$v'_0$}
\end{pspicture}  &
\begin{pspicture}(0,0)(4.4,3.2)
\psline(0,0.4)(4.4,0.4)\psline(0,1.6)(4.4,1.6)\psline(0,2.8)(4.4,2.8)\psline(0.4,0)(0.4,3.2)\psline(1.6,0)(1.6,3.2)\psline(2.8,0)(2.8,3.2)\psline(4,0)(4,3.2)
\psset{linestyle=dashed}
\psline(2,0.4)(2,2.8)\psline(2.4,0.4)(2.4,2.8)
\psline(0.4,0.8)(2.8,0.8)\psline(1.6,2.4)(4,2.4)
\psdots(2,0.8)(2,2.4)(2.4,0.8)(2.4,2.4)
\rput(1.82,0.98){$v_1$}\rput(2.58,0.98){$v_2$}\rput(2.58,2.22){$v_3$}\rput(1.82,2.22){$v_4$}
\rput(1,0.6){$l$}\rput(1.8,0.6){$d$}\rput(2.2,0.6){$d$}\rput(2.6,0.6){$d$}
\end{pspicture} \\
$S''_i$ & $S'''_i$  \\
\end{tabular}
\caption{Figure for the case 2.d. \label{comp}}
\end{center}
\end{figure}
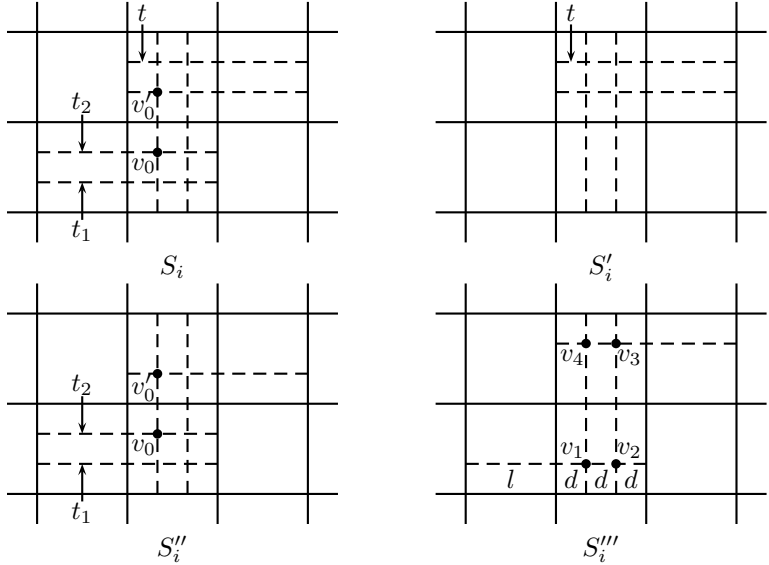

\begin{lemma}\label{+++}
For $\overline{\mathbf{S}}^3(\mathscr{T})$ over a  hierarchical T-mesh of $3\times3$ division, we have
$$
\dim W[T]=\sum_{i=1}^{n}\dim W[T_i].
$$
\end{lemma}
\begin{proof}
Suppose $T_1$ has only one connected component. We will prove this lemma for all of the possible cases of $T_1$.
For case 1, case 2.d, and case 2.e, the conclusion is right according to  Lemma \ref{surject1}.
For case 2.a, case 2.b, case 2.c, and case 3, we can use the same method in the proof of Lemma \ref{dimT_i3}. Here, we only prove
this conclusion for case 2.b. In the same way as in the proof of Lemma \ref{+} and \ref{++}, we construct a new mesh $\mathscr{T}'$
(Figure \ref{last}.a is the mesh excluding $T'$) generated by $T$ and the edges of the $2\times 2$ neighbor cells. We choose the mesh in Figure \ref{last}.b as $\mathscr{T}'_{\bigotimes}$,
and use $L$ to denote the set of all of the l-edges of $T$
not contained in $\mathscr{T}'_{\bigotimes}$. $T'$ is a subset of $L$.
According to Lemma \ref{tensor} and Equation \eqref{lowbound}, we have $$\dim(\mathscr{T}')=1+\dim W[L]
\geqslant 1+\dim W[T'].$$
Conversely, $$\dim(\mathscr{T}')\leqslant \dim W[T]\leqslant \dim W[T_1]+\dim W[T']=1+\dim W[T'].$$
Therefore, $\dim W[T]=1+\dim W[T']
=\dim W[T_1]+\dim W[T']$.
\end{proof}

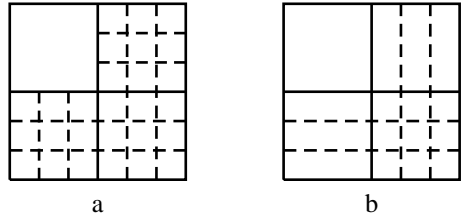
\begin{figure}[!ht]
\begin{center}
\psset{unit=1.3cm,linewidth=1pt}
\begin{tabular}{c@{\hspace{1.3cm}}c}
\begin{pspicture}(0,0)(1.8,1.8)
\psline(0,0)(1.8,0)(1.8,1.8)(0,1.8)(0,0)
\psline(0,0.9)(1.8,0.9)\psline(0.9,0)(0.9,1.8)
\psset{linestyle=dashed}
\psline(0,0.3)(1.8,0.3)\psline(0,0.6)(1.8,0.6)
\psline(1.2,0)(1.2,1.8)\psline(1.5,0)(1.5,1.8)
\psline(0.3,0)(0.3,0.9)\psline(0.6,0)(0.6,0.9)
\psline(0.9,1.2)(1.8,1.2)\psline(0.9,1.5)(1.8,1.5)
\end{pspicture}  &
\begin{pspicture}(0,0)(1.8,1.8)
\psline(0,0)(1.8,0)(1.8,1.8)(0,1.8)(0,0)
\psline(0,0.9)(1.8,0.9)\psline(0.9,0)(0.9,1.8)
\psset{linestyle=dashed}
\psline(0,0.3)(1.8,0.3)\psline(0,0.6)(1.8,0.6)
\psline(1.2,0)(1.2,1.8)\psline(1.5,0)(1.5,1.8)
\end{pspicture} \\
a & b
\end{tabular}
\caption{Figure for the proof of lemma \ref{+++}. \label{last}}
\end{center}
\end{figure}

\begin{theorem}
Suppose $\mathscr{T}$ is a hierarchical T-mesh of $3\times 3$ division with $V^+$ crossing-vertices, $E$ l-edges,
and there are $\delta_{1i},\delta_{2i},\delta_{3i},\delta_{4i}$ connected components of case 2.a, case 2.b, case 2.c, and case 3 in $T_i$, respectively.
$\delta_k=\sum_{i=1}^n \delta_{ki}$, $k=1,2,3,4$. Then, we have
$$
\dim \overline{\mathbf{S}}^3(\mathscr{T})= V^+-2E+4+4\delta_1+2\delta_2+\delta_3+\delta_4.
$$
\end{theorem}
\begin{proof}This is directly derived from the above lemmas.
\end{proof}

Similar to Theorem \ref{nonhbc2}, we have the following theorem.
\begin{theorem}Suppose $\mathscr{T}$ is a hierarchical T-mesh of $3\times3$ division with $V^+$ crossing-vertices, $V^b$ boundary vertices and
$E$ interior l-edges. There are $\delta_{1i},\delta_{2i},\delta_{3i},\delta_{4i}$ connected components of case 2.a, case 2.b, case 2.c, and case 3 in $T^{\varepsilon}_i$,
respectively, where $T^{\varepsilon}_i$ is the i-th level of the extended mesh $\mathscr{T}^{\varepsilon}$. $\delta_k=\sum_{i=1}^n \delta_{ki}$, $k=1,2,3,4$. Then, we have
$$
\dim {\mathbf{S}}^3(\mathscr{T})=3V^b+V^+-2E+4+4\delta_1+2\delta_2+\delta_3+\delta_4.
$$
\end{theorem}

Now, we give an example for the dimension computation.
\begin{example}
In Figure \ref{examplefordim}, $V^+=29$, $V^b=20$, $E=18$, $\delta_{31}=1$. Hence,
$\dim {\mathbf{S}}^3(\mathscr{T})=3\times 20+29-2\times 18+4+1=58$.
\end{example}

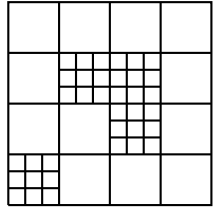
\begin{figure}[!ht]
\begin{center}
\psset{unit=0.75cm,linewidth=0.8pt}
\begin{pspicture}(0,0)(3.6,3.6)
\psline(0,0)(3.6,0)(3.6,3.6)(0,3.6)(0,0)
\psline(0,0.9)(3.6,0.9)\psline(0,1.8)(3.6,1.8)\psline(0,2.7)(3.6,2.7)
\psline(0.9,0)(0.9,3.6)\psline(1.8,0)(1.8,3.6)\psline(2.7,0)(2.7,3.6)
\psline(0.9,2.1)(2.7,2.1)\psline(0.9,2.4)(2.7,2.4)\psline(2.1,0.9)(2.1,2.7)\psline(2.4,0.9)(2.4,2.7)
\psline(1.2,1.8)(1.2,2.7)\psline(1.5,1.8)(1.5,2.7)\psline(1.8,1.2)(2.7,1.2)\psline(1.8,1.5)(2.7,1.5)
\psline(0.3,0)(0.3,0.9)\psline(0.6,0)(0.6,0.9)\psline(0,0.3)(0.9,0.3)\psline(0,0.6)(0.9,0.6)
\end{pspicture}
\caption{A hierarchical T-mesh $\mathscr{T}$. \label{examplefordim}}
\end{center}
\end{figure}
\section{Conclusions and Future Work}

In this paper, the smoothing cofactor-conformality method is explored. We present a new proof of the dimension formula of
$\mathbf{S}(2,2,1,1,\mathscr{T})$ over hierarchical T-meshes and obtain the dimension formula of $\mathbf{S}(3,3,2,2,\mathscr{T})$
over hierarchical T-meshes under the condition that $N(t)\geqslant 2$.
As an application of this method, we
give the dimension formula of $\mathbf{S}(3,3,2,2,\mathscr{T})$ when $\mathscr{T}$ is a hierarchical T-mesh of $3\times 3$ division.

Although most contents in this paper are discussing hierarchical T-meshes, many conclusions also apply to the general T-meshes.
For a general T-mesh, it is not very easy to determine whether all the l-edges have a reasonable order.
Therefore, these conclusions only make sense for some special T-meshes.
For a hierarchical T-mesh, this method is effective when we can compute the dimensions level by level.
Although we only discuss $\mathbf{S}(2,2,1,1,\mathscr{T})$ and $\mathbf{S}(3,3,2,2,\mathscr{T})$,
it is not very difficult to compute $\dim \mathbf{S}(m,n,m-1,n-1,\mathscr{T})$ over hierarchical T-meshes with some restrictions.
However, if we can not compute the dimensions level by level, this method is invalid. We have discussed this phenomenon at the beginning of
Section \ref{3322}.

For the dimension of $\mathbf{S}(3,3,2,2,\mathscr{T})$ over hierarchical T-meshes without any restrictions, the proposed smoothing cofactor-conformality method
 may be invalid. The solution to this problem will be left to future study.

\section*{Acknowledgements}

The authors are supported by a NKBRPC (2011CB302400), the NSF of China (11031007 and 11371341),
and the 111 Project (No. b07033).


\bibliographystyle{elsarticle-num}
\bibliography{survey}

\end{document}